\documentclass[a4paper, 12pt]{article}

\usepackage{lipsum}
\usepackage{amsfonts}
\usepackage{epstopdf}

\usepackage{datetime}

\usepackage[left=1in, top=1in, bottom=1in, right=1in]{geometry}

\usepackage{hyperref} 
\usepackage{academicons}
\usepackage{xcolor}

\usepackage{amsmath}
\usepackage{amsthm}
\allowdisplaybreaks
\usepackage{amssymb}
\usepackage{commath}
\usepackage{mathtools}
\usepackage{bbm}
\usepackage{bm}
\usepackage{mathabx}

\usepackage{color}
\usepackage{graphicx}
\usepackage[small]{caption}
\usepackage{subcaption}
\usepackage{xcolor}
\definecolor{iceblue}{RGB}{150,187,217} 
\definecolor{ocean}{RGB}{0,76,106}
\definecolor{lasergreen}{RGB}{4,185,118} 

\usepackage{relsize}
\usepackage{adjustbox}
\usepackage{algorithmic}
\usepackage{booktabs}
\usepackage{tikz} 
\usetikzlibrary{positioning,arrows.meta,decorations.pathreplacing}
\usepackage{pifont}

\usepackage{comment}

\theoremstyle{definition}
\newtheorem{definition}{Definition}

\theoremstyle{theorem}
\newtheorem{lemma}{Lemma}
\newtheorem{theorem}{Theorem}

\theoremstyle{remark}
\newtheorem{example}{Example}

\newtheorem{remark}{Remark}

\title{Entropy-aware non-oscillatory high-order finite volume methods using the Dafermos entropy rate criterion}
\date{Received: date / Accepted: date}
\author{ 
Simon-Christian Klein\thanks{Department of Mathematics, TU Braunschweig, Braunschweig, 38106, Germany (simon-christian.klein@tu-braunschweig.de)} 
\and 
Thomas Sonar\thanks{Department of Mathematics, TU Braunschweig, Braunschweig, 38106, Germany (t.sonar@tu-bs.de)} 
}

\usepackage{amsopn}

\usepackage[normalem]{ulem}

\DeclareMathOperator*{\argmin}{arg\,min}

\newcommand{\intd}{\, \mathrm{d}}

\newcommand{\R}{\mathbb{R}}

\newcommand{\vd}{\mathrm{d}} 
\newcommand{\derd}[2]{\frac{\vd #1}{\vd #2}}
\newcommand{\der}[2] {\frac {\partial #1}{\partial #2}}
\newcommand{\fve}{E^{\mathrm{FV}}}
\newcommand{\fzul}{\mathcal{F}}
\newcommand{\uzul}{\mathcal{U}}
\newcommand{\ch}{\mathrm{conv}}

\newcommand{\eavs}{\overgroup{*}}
\renewcommand{\epsilon}{\varepsilon}
\newcommand{\diam}{\mathrm{diam}}
\DeclareMathOperator{\Iop}{\mathcal{I}}
\newcommand{\skp}[2]{\left \langle {#1} \middle | {#2} \right \rangle}

\DeclareMathOperator{\cent}{cent}
\DeclareMathOperator{\radi}{rad}
\DeclareMathOperator{\B}{\mathcal{B}}

\DeclareMathOperator{\TV}{TV}
\newcommand{\Recon}{\mathcal{R}}
\newcommand{\Sradi}{\mathcal{S}}
\newcommand{\umean}{\Iop u}
\newcommand{\Rest}{\mathrm{R}}
\DeclareMathOperator{\id}{Id}
\newcommand{\rerror}{\mathcal{E}}
\newcommand{\be}{\begin{equation}}
\newcommand{\ee}{\end{equation}}

\newcommand{\OptTwidth}{0.9\textwidth}

\begin{document}

\maketitle

\begin{abstract}
	Finite volume methods are popular tools for solving time-dependent partial differential equations, especially hyperbolic conservation laws. 
Over the past 40 years a popular way of enlarging their robustness was the enforcement of global or local entropy inequalities. 
This work focuses on a different entropy criterion proposed by Dafermos almost 50 years ago, stating that the weak solution should be selected that dissipates a selected entropy with the highest possible speed. We show that this entropy rate criterion can be used in a numerical setting if it is combined with the theory of optimal recovery.
To date, this criterion has only seen limited use in Finite-Volume schemes and to the authors knowledge this work is the first in which this criterion is applied to a Finite-Volume scheme whose accuracy is based on reconstruction from mean values. This leads to a new family of schemes based on reconstruction providing an alternative to the popular ENO and WENO schemes.
	
\end{abstract}

\section{Introduction} 
\label{sec:intro} 

Since the first algorithms for hyperbolic conservation laws  
were introduced \cite{NR1950Method,Lax1954LF,Richtmyer1957,Godunov1959Difference} one of the main questions that had to be answered related to the question whether the problems concerning uniqueness and their resolution \cite{Kruzkhov1960,Kruzhkov1964,Oleinik1963} 
could be also carried over to the numerical setting. This was done by calculating approximate solutions that not only satisfy a numerical conservation law but also a numerical entropy inequality as for example the Lax-Friedrichs and Godunov schemes \cite{Lax1971Shock,Tadmor1984I,Tadmor1984II}. Such an entropy (in)equality can be proved also for other approximate numerical solvers, for example approximate Riemann solvers \cite{HLL1983} and Tadmor introduced a theory to design schemes and their numerical fluxes from the beginning to satisfy entropy (in)equalities \cite{Tadmor1987The}. This theory was used and extended to design the nowadays ubiquitous entropy dissipative high order schemes \cite{LMR2002Fully,FC2013HO,FMT2012AO,KO2022Entropy,WKGH2021Construction,RSCRWHG2021}. 
Sadly, counterexamples for the uniqueness of solutions to the Euler equations satisfying the entropy equality as an example for a system of conservation laws in several space variables exist \cite{Chiodaroli2014} and one is therefore tempted to study the application of alternative entropy criteria to single out nonphysical solutions.

Dafermos proposed a different entropy criterion \cite{Dafermos72} for which he showed that it recovers the unique solution under all piecewise smooth solutions. The criterion states that a weak solution should be selected from all possible ones that dissipates a selected entropy as fast or faster than all other weak solutions. 
While this criterion attracted theoretical attention \cite{Feireisl2014MD} it is difficult to use in a numerical setting, as the set of all weak solutions is not available to the scheme's designer. Instead (weak) solutions of a numerical conservation law are considered that are only approximate solutions to the original conservation law. Clearly, one can devise schemes that produce solutions to numerical conservation laws that are as dissipative as one wishes. This corresponds to larger errors with respect to the original conservation law.

To the authors knowledge the first connection to the numerical analysis of conservation laws is the proof that for scalar conservation laws the Godunov flux can be recovered from the criterion by assuming that the numerical flux used is given by the evaluation of the continuous flux at a point in the convex hull between the value on both sides of the cell interface \cite{Dafermos2009MDR,Dafermos2012MD}. It should be noted that this theorem does not hold for the Godunov flux for general systems of conservation laws \cite{ranocha2018thesis}.

The idea to restrict the amount of possible deviation from the conservation law was also used in \cite{klein2022using,klein2022stabilizing}. In the first case the additional knowledge that the entropy equality holds for smooth solutions was used, while in the second case the maximal dissipation idea was applied to the internal degrees of freedom of a discontinuous Galerkin method.

We will generalize this point of view in this work to finite volume methods using reconstruction. A critical piece in our construction is the theory of optimal recovery and the notion of the radius of information as a restriction to the flux. In essence, our numerical flux will also stem from the evaluation of the continuous flux, but our restriction will be more intricate than in the Godunov case and will be provided by this radius of information.

The next section, section \ref{sec:prelim}, will give a general overview of the previous work needed for the construction, starting with conservation laws, recovery based finite-volume methods, the theory of optimal recovery, and culminating in the entropy rate criterion.

In section \ref{sec:ourapproach} our construction is explained in several steps. First, the fluxes used by us are defined and a new theorem concerning the Lax-Friedrichs flux is proved. This theorem states that the flux can be interpreted as an approximate solution to the variational problem defining a flux in \cite{Dafermos2009MDR}. Next, our algorithms for the restriction of the variational problem stated by these fluxes is presented, i.e. our way of guessing the radius of information. The last subsection is devoted to a regularization procedure, as the fluxes designed up to that point do not posses a bounded viscosity, but this property is necessary for well behaved schemes \cite{Tadmor1984I,Tadmor1984II,Harten1978The,HARTEN1983Aclass}. This regularization is also catered to respect Dafermos entropy rate criterion.

Numerical tests in section \ref{sec:tests} are used to verify the usefulness of our theoretical findings. We test our different algorithms for the guessing of the radius of information for smooth and discontinuous solutions. These tests will also allow us to verify the high accuracy of most combinations between variational flux and information radius predictor. We close our observations with a conclusion distilling the most important findings of this publication. 

\section{Preliminaries} 
\label{sec:prelim}

\subsection{Hyperbolic conservation laws} 
We are interested in solutions to systems of conservation laws in one space dimension, given by \cite{Smoller1994Shock}
\begin{equation} \label{eq:HCL}
    \der{u}{t} + \der{f(u(x, t))}{x} = 0.
\end{equation}
Here, $u: \R \times \R \to \R^k$ is the vector valued function of conserved variables, while $f: \R^k \to \R^k$ is the flux function. We call a system strict hyperbolic if the matrix
$
    \derd{f}{u}
$
has $k$ distinct eigenvalues $\lambda_1 < \lambda_2 < \dots < \lambda_k$ and corresponding eigenvectors \cite{Lax1973Hyperbolic}. Smooth initial conditions to \eqref{eq:HCL} can develop discontinuities in finite time \cite{Lax1973Hyperbolic}. We therefore search for weak solutions \cite{GR91Hyperbolic}, i.e. functions $u$ that satisfy \eqref{eq:HCL} in the sense of distributions \cite{Lax1971Shock}. Sadly, these are not necessarily unique, and one therefore hopes that additional restrictions to the space of admissible solutions single out the physically relevant, or at least one physically relevant solution. A large class of these additional criteria are the entropy conditions. A classical entropy condition is given by a convex functional $U:\R^k \to \R$ and a suitable entropy flux $F:\R^k \to \R$ \cite{Lax1971Shock,FL1971Systems}, satisfying the compatibility relation
\[
    \derd{U(u)}{u} \derd{f(u)}{u} = \derd {F(u)} {u}.
\]
One can show that this relation implies that smooth solutions satisfy an additional conservation law, the entropy equality \cite{Lax1971Shock,FL1971Systems}
\[
    \der{U \circ u(x, t)}{ t} + \der{F\circ u(x, t)}{x} = 0,
\]
where the conserved quantity is the convex functional $U$ and the flux is given by the previously defined entropy flux $F$. If a conservation law is regularized with viscosity of strength $\epsilon$, i.e. if
\[
    \der{u_\epsilon(x, t)}{t} + \der{f(u_\epsilon(x, t))}{x} = \epsilon \der{^2u_\epsilon(x,t)}{x^2}
\]
is solved instead of \eqref{eq:HCL} and the limit $u_\epsilon \xrightarrow{\epsilon \to 0} u$ exists, the limit is called a vanishing viscosity limit \cite{Lax1971Shock}. Because the entropy is assumed to be convex it follows
\be
\label{eq:eieq}
\der{U(u(x, t))}{t} + \der{F(u(x, t))}{x} \leq 0
\ee
in this case for all non-negative test functions in the sense of distributions \cite{Lax1971Shock}.
One therefore requires from a weak solution that it also satisfies \eqref{eq:eieq} and calls the solutions that satisfy this criterion entropy dissipative solutions. 

\subsection{Finite-Volume Methods}
A fruitful method for the construction of numerical schemes for hyperbolic conservation laws are finite volume (FV) methods \cite{MS2007FVM}. 

One starts with the domain of interest $\Omega$, an interval $[x_L, x_R]$ on the real line in our case, and subdivides the domain into subdomains $(\omega_k)_k$, called cells. We will denote these sub-intervals by $\left[x_{k-\frac 1 2}, x_{k+\frac 1 2}\right]$, their maximum length as 	$\Delta x = \max_{1 \leq n \leq N} (x_{n+1} - x_n)$ and the center of each interval as $x_k$.
If one integrates the given hyperbolic conservation law over such a finite domain in space $\left [x_{k-\frac  1 2}, x_{k+\frac 1 2}\right]$ and a time interval $\left [t_n, t_{n+1} \right]$ one finds
\begin{equation}\label{eq:ICL}
\begin{aligned}
   0 &= \int_{t_n}^{t_{n+1}} \int_{x_{k-\frac 1 2}}^{x_{k+\frac 1 2}} \der {f \circ u}{x} \intd x \intd t + \int_{x_{k-\frac 1 2}}^{x_{k+\frac 1 2}} \int_{t_1}^{t_2} \der{u}{t} \intd t \intd x \\
   &= \int_{t_n}^{t_{n+1}} f\left(x_{k+\frac 1 2}, t\right) - f\left(x_{k-\frac 1 2}, t\right) \intd t + \int_{x_{k-\frac 1 2}}^{x_{k+\frac 1 2}} u\left(x, t_{n+1}\right) - u\left(x, t_{n}\right) \intd x.
\end{aligned}
\end{equation}
When the mean values at time $t_n$ in cell $k$ are denoted by
\[
    u^n_{k} = \frac{ 1}{x_{k+\frac 1 2} - x_{k-\frac 1 2}} \int_{x_{k-\frac 1 2}}^{x_{k+\frac 1 2}} u\left(x, t_n\right) \intd x
\]
the update formula
\begin{equation}
    \label{eq:FVupdate}
    u^{n+1}_{k} = u^{n}_k + \frac{ \int_{t_n}^{t_{n+1}} f\left(x_{k-\frac 1 2}, t\right) \intd t - \int_{t_n}^{t_{n+1}} f\left(x_{k+\frac 1 2}, t\right) \intd t}{x_{k+\frac 1 2} - x_{k-\frac 1 2}}
\end{equation}
follows by a rearrangement of \eqref{eq:ICL}. While this formula can be interpreted without problems as the change of the total amount of $u$ in the subdomain around $x_k$ being equal to the flux of the conserved variables over the boundary of the cell, the implementation is not so simple. While we have an update formula for the mean values the flux over the cell boundaries over a finite time interval is a needed ingredient in this formula. Our main task is therefore the construction of suitable approximate fluxes across the boundaries and a first step is dividing \eqref{eq:FVupdate} by $t_{n+1} - t_n$ and passing to the limit $t_{n+1} \to t_n$. One therefore reduces the problem to the semi-discrete setting
\[
    \derd{u_k}{t} = \frac{f\left(x_{k-\frac 1 2}, t\right) - f\left(x_{k+\frac 1 2}, t\right)}{x_{k+\frac 1 2} - x_{k-\frac 1 2}},
\]
i.e. approximating the flux not over the complete time interval, but only at a certain time. Still one needs to devise a way of advancing from the mean values approximate fluxes, and an important building block are reconstructions that allow us to use point values while we only advance mean values in time.
\subsection{Reconstruction from cell average values} \label{sec:fvmrecon}
Van Leer \cite{Leer1973TowardsI,Leer1974TowardsII,Leer1977TowardsIII,Leer1977TowardsIV,Leer1979TowardsV} introduced his highly successful up to second order MUSCL scheme. Later his ideas were generalized \cite{CW1984PPM,Harten1987ENOIII,LOC1994WENO} to principally arbitrary order. See also \cite{NT1990NOC} for a different kind of scheme using reconstructions. A fundamental step in all of these schemes is the reconstruction of an approximate function $u(\cdot, t_n)$ from average values $(u^n_k)_k$. In these classical methods, the approximate function $u(\cdot, t_n)$ is a polynomial in every cell but can be discontinuous between cells. To determine a polynomial $p$ for cell $k$ a stencil including several neighboring cells $S = \set{k-l, k-l+1, \ldots, k, \ldots, k+m}$ is selected. A polynomial $p$ of order $r$ is constructed that satisfying
\[
    \forall i \in S: \frac{1}{x_{i+\frac 1 2} - x_{i-\frac 1 2}} \int_{x_{i-\frac 1 2}}^{x_{i+\frac 1 2}} p(x) \intd x = u^n_i,
\]
i.e. the average values of the polynomial correspond to the average values of the function that we want to recover in the selected cells. The choice $r = 0$ corresponds to the interpretation of the average values as piecewise constant functions. The MUSCL, ENO and WENO methods differ in the way the stencil $S$ is selected and how the reconstructed polynomial $p$ is used. 
\begin{enumerate}
    \item MUSCL reconstructs two polynomials of order $1$ with the left and right neighbour. If both reconstructions have the same sign of the slope the one with lower absolute slope is used, otherwise they are discarded and a trivial reconstruction of order $0$ is used. Afterwards this reconstructed solution is entered into an generalized Riemann solver \cite{BA2011GRP}, i.e. the flux over the cell edge is approximated by the flux in the solution of the Riemann problem to this piecewise polynomial initial condition.
    \item ENO methods recursively enlarge the stencil of the reconstruction until the size of the stencil reaches the desired size. The standard procedure for this stencil enlargement is to calculate the highest divided differences of the stencil enlarged to the left and right and choose the cell with the smaller absolute value of the divided differences. Afterwards the mean values of this piecewise polynomial initial condition are carried on in time using a Cauchy-Kowaleskaya procedure \cite{Harten1987ENOIII} or Runge-Kutta time integration \cite{SO1988EfficientI,SO1989EfficientII}.
    \item WENO methods reconstruct polynomials from all connected stencils of the selected size. Afterwards the smoothness of the solutions on these stencils is estimated using smoothness indicators. The end result of the computation is a convex combination of the reconstructions weighted by their smoothness, with the property that a high order reconstruction is constructed out of all polynomials for smooth solutions while for a non-smooth solution only smooth polynomials contribute to the final polynomial. Afterwards, these are used as in the ENO-Type methods.
\end{enumerate}

\subsection{The theory of optimal recovery} \label{sec:optirecov}
The reconstruction of an approximate solution from the calculated mean values is a key element to achieve high orders of accuracy in finite volume solvers \cite{Sonar1997ENO}. Sonar discovered that this procedure can be interpreted as the problem of (optimal) recovery \cite{Sonar1996Optimal,Sonar1997MENO}. As we are not only interested in reconstructions but also their errors and will use this information in our numerical fluxes we will give a short overview of the theory of optimal recovery, as initiated by Golomb and Weinberger \cite{GW1959Optimal}, see also \cite{MR1978Survey,MR1980Optimal,MR1984Lectures}. 
We will assume in the following that $V$ is a Banach space with a continuous point functional $\delta$, and the function we would like to recover $u$ lies in exactly this space $V$. To model the limited knowledge available on the function $u$ we define an information operator.
\begin{definition}[Information operator]
	A linear map 
	\[
	\Iop: V \to \R^n, u \mapsto \Iop(u) = \begin{pmatrix}
			\Iop_1(u) \\ 
			\vdots \\
			\Iop_n(u)
		\end{pmatrix},
		\]
		mapping a function $u$ to a vector of real numbers, the given information, will be called an information operator.
	\end{definition}
In our case the information operator is the cell average value operator, as our known information are the cell average values
\[
	\Iop_k(u(\cdot, t)) = \frac{1}{x_{k +\frac 1 2} - x_{k-\frac 1 2}} \int_{x_{k - \frac 1 2}}^{x_{k + \frac 1 2}} u(x, t) \intd x.
\]
It is often possible to restrict the part of the space $V$ in which the function $u$ resides, and it would be unwise to ignore this information as one can often find functions with arbitrary distance to our desired function when no restrictions are imposed. We therefore search for $u$ in what we call the admissible set, that results from a restriction of $\Iop$.
\begin{definition}[Admissible set]
	
	$W \subset V$ is called an admissible set, if $W$ is convex and 
	\[
		\forall u \in V\, \exists c > 0: cu \in U \text{ and } u \in W \implies -u \in W
	\]
	holds. An (not necessarily linear) operator $\Rest: V \supset W \to V$ is called a restriction operator, if the set
	\[
		W = \set{u \in V| \norm{\Rest u} \leq 1 }
	\]
	is an admissible set. 
	\end{definition}
One possible way to restrict the admissible region, at least for hyperbolic conservation laws, is entropy as the following example shows.
\begin{example}[Entropy as restriction]
		Given a scalar hyperbolic conservation law with entropy $U = \frac{u^2}{2}$, the set
		\[
		W= \left \{u \in V  \middle | \int U(u) \intd x \leq M \right\} 
		\]
		is convex and satisfies the additional requirements of an admissible set for $M > 0$. Further, a restriction operator generating this admissible set in conjunction with the $2$ norm is the identity $\Rest = \id$. 
	\end{example}

\begin{definition}[Reconstruction operator]
	We call an not necessarily linear operator 
	\[
	\Recon(x): \R^n \to \R, v \mapsto \Recon(x) v 
	\] 
	 a reconstruction operator, if this operator is used to predict point values of the function $u$.
	\end{definition}
One would therefore wish for $\Recon \Iop = \delta_x \iff (\Recon \Iop u)(x) = u(x)$, i.e. that the reconstruction is a left inverse to the information operator. Clearly, this is not possible in general. Therefore the notion of error is important to reduce the impact of the uncertainties that enter through this problem.
\begin{definition}[Error of a reconstruction]
	The worst case error of a reconstruction is defined as
	\[
	\rerror_{\Recon(x)}(\Iop, W) = \sup_{u \in W} \abs{u(x) - \Recon(x) \Iop u }.
	\]
	\end{definition}
Clearly, an optimal recovery can not in every case have a zero worst case error as this would restrict the set $ W \subset V$ to be so small that it would be of trivial interest. On the other hand, the worst case error can be infinite if $W$ is chosen to large. We therefore define the intrinsic error as the best possible one.
\begin{definition}[Intrinsic error]
	For a given admissible set $W \subset V$ and information operator $\Iop$ is the intrinsic error of this combination $(\Iop, W)$ given as
	\[
	\rerror(\Iop, W) = \inf_{R(x)} \rerror_{\Recon(x)}(\Iop, W).
	\]
	\end{definition}
An operator with exactly this intrinsic error will be called an optimal reconstruction operator. We can recover some of our intuition from the geometry of $\R^2$ by defining the diameter and radius of a set $A \subset \R$
\[
	\diam(A) = \sup_{a, b \in A} \abs{a-b}, \quad \radi(A) = \inf_{y \in \R} \sup_{a \in A} \abs{y-a}.
\]
Further, an element $c \in \cent V$ is said to be a center of a set $A$ if 
\[
	\sup_{a \in A} \abs{c - a} = \radi(A)
\]
holds.
Clearly, if one defines
\[ 
B(\Iop u) = \set{v \in W | \Iop u = \Iop v},
\]
the set of all functions in $W$ that share the same information as $u$, then
\[
\diam(x, \Iop, W) = \sup_{u \in W} \sup_{v \in B(\Iop u)} \abs{v(x) - u(x)}
\]
can be thought of as the diameter and
\[
\radi(x, \Iop, W) = \sup_{u \in W} \inf_{y \in \R} \sup_{v \in B(\Iop u)} \abs{y - v(x)}
\]
as the radius of the information operator. This is even more clear if one defines
\[
	\mathcal{V}(\Iop u) = \set{v(x) | v \in B(\Iop u)},
\]
the set of all possible values for all functions with the same information, as then
\[
		\diam(x, \Iop, W) = \sup_{u \in W} \diam(\mathcal{V}(\Iop u)), \quad \radi(x, \Iop, U) = \sup_{u \in W} \radi(\mathcal{V}(\Iop u))
\]
hold. A result from this definition is the following theorem.
\begin{theorem}[Error and diameter] \label{thm:optirad}
	The error of an optimal recovery operator is equal to the radius of information, i.e.
		\[
			\rerror_{\Recon(x)} = \radi(x, \Iop, W).
		\]
\end{theorem}
\subsection{The entropy rate criterion}
In \cite{Dafermos72} an entropy rate criterion was introduced to reduce the number of admissible weak solutions. Central to this entropy rate condition is the total entropy $E_u$ associated with a solution $u$,
\[
    E_{u}(t) = \int U(u(x, t)) \intd x.
\]
A weak solution $u$ is said to satisfy Dafermos' entropy rate criterion if for all other weak solutions $\tilde u$ it holds
\[
    \forall t \geq 0:\, \derd{E_{u}}{t} \leq \derd{E_{\tilde u}}{t}.
\]
In other words: the selected weak solution should dissipate entropy equally fast or faster than all other weak solutions. 
This entropy criterion is able to single out nonphysical solutions to the Euler equations that are not singled out by the usual entropy condition for the physical entropy \cite{Feireisl2014MD}. Further, it was shown \cite{Dafermos2009MDR,Dafermos2012MD} that for scalar conservation laws the classical Godunov flux using the entropy solution of the Riemann problem is characterized by this entropy rate criterion in the following sense. Let us define the total Finite-Volume entropy as
\[
    \fve = \fve_{\left(u_k\right)_k}(t) = \sum_{k} U\left(u^n_k(t)\right) \Delta x_k.
\]
The change of this total entropy depends on the scheme used, and therefore on the inter-cell fluxes $f_{k+\frac 1 2}$. Together with the entropy variables $\derd U u$ one therefore concludes that
\[
    \derd{\fve}{t} = \sum_k \skp{\derd U u(u_k)}{ \derd {u_k} t} \Delta x_k = \sum_k \skp{\derd U u(u_k)}{ f_{k-\frac 1 2} - f_{k+\frac 1 2}}
\]
is the change of the total entropy with time. As this should be as negative as possible every flux $f_{l+\frac 1 2 }$ has to satisfy
\[
    \begin{aligned}
    f_{l+\frac 1 2} = \argmin_{f_{l + \frac 1 2}}& \derd{\fve}{t} \\
    = \argmin_{f_{l + \frac 1 2}} & \underbrace{\sum_{k \neq l, l+1} \skp{\derd U u(u_k)}{f_{k-\frac 1 2} - f_{k+\frac 1 2}}}_{\text{const. w.r.t. } f_{l + \frac 1 2}}
    \\
    	\phantom{\argmin_{f_{l + \frac 1 2}}} &+\skp{\derd U u (u_l)}{ f_{l-\frac 1 2} - f_{l + \frac 1 2}} 
    	+ \skp{\derd U u (u_{l+1})}{ f_{l + \frac 1 2}- f_{l+\frac 3 2}} \\
    =\argmin_{f_{l + \frac 1 2}}& \skp{\derd U u (u_{l+1}) - \derd U u (u_{l})}{ f_{l + \frac 1 2}} \\
    	 &+ \underbrace{ \skp{\derd U u(u_l)}{ f_{l-\frac 1 2}} + \skp{\derd U u(u_{l+1})}{ f_{l+\frac 3 2} } }_{\text{const. w.r.t. } f_{l + \frac 1 2}} \\
    = \argmin_{f_{l + \frac 1 2}}& \skp{\derd U u (u_{l+1}) - \derd U u (u_{l})}{ f_{l + \frac 1 2}},
    \end{aligned}
\]
as this is a necessary condition. We did not yet restrict the set $\fzul$ of admissible values for $f_{l + \frac 1 2}$ and this will be the core consideration of this work. In \cite{Dafermos2012MD} the equivalence
\[
    f^{\rm G}(u_l, u_r) = \argmin_{u \in \ch(u_l, u_r) } \skp{\derd U u (u_{r}) - \derd U u (u_{l})}{ f(u)} 
\]
was established for scalar conservation laws, i.e. the solution to the scalar Godunov intercell flux can be recovered from the entropy rate criterion. In this case $\fzul = \{f(u) | u \in \ch(u_l, u_r) \}$ has to be chosen.
\section{Our approach} \label{sec:ourapproach}
Our approach is based on a generalization of the previous observation of Dafermos that the Godunov flux can be recovered from the maximum entropy rate criterion. We define a new intercellular flux as the most dissipative flux in relation to a set of possible flux values. Our set of admissible flux values will in turn be based on the error of a reconstruction which we employ in our scheme. This error will be described by a redefinition of the radius of information.
\begin{definition}[Dafermos' flux]
	We define \emph{Dafermos' flux with set of admissible fluxes (short admissible flux set) $\fzul$} as 
	\begin{equation}
	f^{\rm D}_\fzul(u_L,u_R) = \argmin_{f \in \fzul} \skp{\derd U u(ur) - \derd U u (u_l)}{f}.
	\end{equation} 
\end{definition}
\begin{definition}[Dafermos' flux using admissible conserved variables]
	We define \emph{Dafermos' flux with a set of admissible conserved values (short admissible conserved set) $\uzul$} as 
	\begin{equation}
		f^{\rm D}_\uzul(u_L,u_R) = f\left(\argmin_{u \in \uzul} \skp{\derd U u(ur) - \derd U u (u_l)}{f(u)} \right).
	\end{equation} 
\end{definition}
The difference of the two previous definitions lies merely in the fact that the first one considers variations of the flux, while the second one considers variations of a value of conserved variables $u$, that is afterwards entered into the flux function. Clearly, the second definition can be brought into the form of the first one by setting $\fzul = f(\uzul) = \set{f(u) | u \in \uzul}$, but the first form allows for sets $\fzul$ that can't be described solely with the second approach. We will define our set of admissible conserved variables $\uzul$ using a localized version of the radius of information defined in the previous section, i.e. this radius will not be the worst case over all possible $u$, but instead be defined for a single chosen $u$.
\begin{definition}[Local information radius]
	The local information diameter shall be defined as
	\[
		\diam(x, \Iop, W, u) = \sup_{v \in B(\Iop u)} \abs{v(x) - u(x)},
	\]
	together with the information radius defined by
	\[
		\radi(x, \Iop, W, \Iop u) = \inf_{y \in \R} \sup_{v \in B(\Iop u)} \abs{v(x) - y} = \radi(V(\Iop u)).
	\]
\end{definition}
It should be noted that the set $V(\Iop u)$ does not depend on the particular $u$ selected, but on its information, i.e. the finite dimensional vector $\Iop u$. We therefore strive to find the radius of information without knowledge of $u$ and will use $\radi(\Iop u) = \radi(V(u))$ as a short form. Our admissible sets $\uzul$ could for example be of the form 
\[
	\uzul = \B_{\radi(\Iop u)}(u_c), \quad  u_c \in \cent{\mathcal{V}(\Iop u)}
\]
i.e. the ball of radius $\radi(\Iop u)$ around a center of the set $\mathcal{V}(\Iop u)$.
While we saw already, that the Godunov flux can be interpreted as such a flux with a certain set $\fzul$ or $\uzul$, one can also interpretate certain approximate Riemann solvers as beeing approximate solutions to this optimization problem. One example for such an approximate Riemann solver is the following flux which is the result of a modification of the classical local Lax-Friedrichs flux \cite{Lax1954LF}.
\begin{definition}[Modified Lax-Friedrichs flux]
	Let $f$ be a smooth flux function. We define the \emph{modified Lax-Friedrichs} flux with admissible set $\uzul$ as
	\[ f^{\rm{MLF}}_\uzul(u_l, u_r) = f\left(u_c\right) + \frac{\derd {U}{u}(u_l) - \derd{U}{u}(u_r)}{\norm{\derd {U}{u}(u_l) - \derd{U}{u}(u_r)}} \max_{u \in \uzul} \norm{\derd f u} \radi(\uzul),
	\]
	where $u_c$ is a center of the set $\uzul$.
\end{definition}
This flux is a slight generalization and modification of the Lax-Friedrichs flux, as $U(u) = u^2/2$ and $\uzul = \ch (u_l, u_r)$ recovers 
\[
	f(u_l, u_r) = f\left(\frac{u_l + u_r}{2}\right) + \max_{u \in U} \norm{\derd f u} \frac{u_l - u_r}{2} 
\]
as value for this flux, giving a modified local Lax-Friedrichs flux. Clearly, this flux is also monotone for scalar conservation laws and can be written in the form of a Dafermos flux with set of admissible flux values $\fzul$ once one defines
\[
	\fzul^{\rm{MLF}}(\uzul) = \left \{f(u_c) + g \middle | \norm g \leq \radi(\uzul) \max_{u \in U} \norm{\derd f u} \right \}.
\] 

This flux can be evaluated significantly easier than Dafermos' flux and is an approximate solution to the optimization problem in the sense that while in general there exists no $u^* \in \uzul$ with
\[
f(u^*) = f_\uzul^{\rm{MLF}}
\]
 this flux is even more dissipative than Dafermos' flux in the sense of the following lemma.

\begin{lemma}
	The modified Lax-Friedrichs flux is more dissipative than the Dafermos flux calculated using the same co-domain
	\[
	\skp{\derd U u (u_r) - \derd U u (u_l)}{ f^{\rm{MLF}}_\uzul} \leq \skp{\derd U u (u_r) - \derd U u (u_l)}{ f^{\rm D}_\uzul}
	\]
	\begin{proof}     Let $u^* \in \uzul$ denote an argument minimizing the functional in the variational problem of Dafermos' flux.  As $u_c$ is a center of the admissible set $\uzul$ \[
		\norm{u_c - u^*} \leq \sup_{u \in \uzul} \norm{u_c - u} = \radi(\uzul)
		\]
		 is satisfied. The flux function is localy Lipschitz-continuous on the admissible set $\uzul$ with constant $L_f = \max_{u \in \uzul} \norm{\derd f u(u)}$. We can therefore conclude that
		\[
			\norm{f(u_c) - f(u^*)} \leq L_f \norm{u_c - u^*} \leq L_f \radi(\uzul)
		\]
		allows us to bound the error in the flux stemming from using the center of $\uzul$ instead of the most dissipative value. The corresponding error in the entropy dissipation is therefore bounded by
		\[
			\begin{aligned}
			\abs{\skp{\derd U u(u_r) - \derd U u(u_l)}{f(u_c) - f(u^*)}} &\leq \norm{\derd U u(u_r) - \derd U u(u_l)} \norm{f(u_c) - f(u^*)}\\
			 &\leq \norm{\derd U u(u_r) - \derd U u(u_l)} L_f \radi(\uzul).
			\end{aligned}
		\]
		As in in \cite{klein2022stabilizing} for DG methods inside elements this contribution can be counteracted by adding a correction to the flux for which we will use the most dissipative direction possible. This is exactly the difference in the entropy variables and yields
		\[
		\begin{aligned}
		&\skp{\derd U u (u_r) - \derd U u (u_l)}{f(u^c) - L_f \radi(\uzul) \frac{\derd U u(u_r) - \derd U u (u_l)}{\norm{\derd U u(u_r) - \derd U u (u_l)}}} \\
		\leq& \skp{\derd U u (u_r) - \derd U u (u_l)}{ f(u^*)} \\
		&+ \underbrace{\norm{\derd U u(u_l)-\derd U u (u_r)}L_f \radi(\uzul) - \norm{\derd U u(u_l)-\derd U u (u_r)}L_f \radi(\uzul)}_{= 0}.
			\end{aligned} 
		\]
		 This shows the claim, as the first expression exactly describes the entropy dissipation of the modified Lax-Friedrichs flux while the second expression in the inequality equals the entropy dissipation of the Dafermos flux. 
	\end{proof}
\end{lemma}
\begin{remark}
	On a side note, the (classical) Lax-Friedrichs flux for scalar conservation laws
	\[
		f^{\mathrm{LLF}}(u_l, u_r) = \frac {1}{2} \left( f(u_l) + f(u_r) + c_b (u_l - u_r) \right) 
	\] can also be interpreted as such an approximate solution when the speed bound $c_b$ satisfies $c_b \geq \max_{u \in \ch(u_l, u_r)} \abs{\derd f u}$, and this is the classical definition of the local Lax-Friedrichs flux. Clearly, choosing $c_b = \frac{\Delta x}{\Delta t} \geq \sup \abs{\derd f u}$ as in the classical Lax-Friedrichs flux is also sufficient. Still we are using the modified Lax-Friedrichs flux presented earlier as it allows us to use the recovered point value of $u$ as the center of the set $\uzul$ together with our prediction for the information radius.
\end{remark}

Our plan will now be to
\begin{itemize}
	\item design recovery procedures that allow us to not only recover approximations for point values, but also estimates on the error like the radius of information
	\item use the output of these recovery procedures as admissible sets $\uzul$ for our numerical fluxes. 
\end{itemize}

\subsection{Choosing the the admissible conserved variables $\uzul$}
In this section several algorithms for the reconstruction of point values from average values in spite of the classical MUSCL, ENO and WENO methods are described. A distinction to classical methods is that we do not want to reconstruct piecewise polynomial solutions on their own, but point values together with uncertainties in the form of the information diameter. The classical notion of a jump at a cell boundary is therefore discarded and replaced by a set of possible values. Still a classical reconstruction can be converted to this new setting by calculating $u_c = \frac 1 2 (u_l + u_r)$ and $\radi(\uzul) = \frac 1 2 \abs{u_r - u_l}$, corresponding to $\uzul = \ch(u_l, u_r)$, from the reconstructed left and right states. The basis of all of these methods are recovery polynomials as described in section \ref{sec:fvmrecon}. In what follows, we will denote the operators recovering point values as $\Recon$, while an information radius predictor will be written as $\Sradi$, not to be confused with stencils $S$.

\subsubsection{Using a weighted variance}
Our first example of a suitable information radius indicator for reconstruction is based on the usage of a weighted variance of the results of different stencils. Let $(\Recon_k)_{k = 1, \dots, K}$ be a sequence of different reconstruction operators. These can be, and in our implementation are, simple polynomial reconstruction operators of the same order as explained in section \ref{sec:fvmrecon} and \ref{sec:optirecov}, constructed on different stencils. We can define their average reconstruction as
\[
    \bar \Recon = \frac{\sum_{k=1}^K \Recon_k}{K}.
\]
We will see that while from a theoretical standpoint such a simple average value would service our needs a weighted mean
\[
	\bar \Recon = \frac{\sum_{k=1}^K w_k \Recon_k}{\sum_{k=1}^K w_k}
\]
is better suited to allow for high fidelity simulations in the presence of discontinuities if suitable weights are used as for example in the WENO \cite{LOC1994WENO} method. The simple mean is enclosed using the trivial weights $w_k =1 $.
When these reconstruction operators are applied to a solution $u$ that is not a polynomial with an order less or equal than the recovery polynomials for which the operators were constructed they will have some variance
\[
    v_k^2 = \frac{(\bar \Recon \umean - \Recon_k \umean)^2}{K}
\]
and we can estimate the variance of their average
\[
    (\bar v)^2 = \frac{\sum_{k=1}^K v_k^2}{K}
\]
using this information. We therefore could use
\[
    \bar \Recon \Iop u, \quad \Sradi \umean = \sqrt{{\bar v}^2}
\]
as a reconstruction with an information radius estimate.

\subsubsection{Using the bounding sphere of all possible reconstructions}
Given a sequence of possible reconstruction operators $(\Recon_k)_{k = 1}^K$ that are applied to a particular set of average values $\Iop u$ one can define the bounding sphere of these reconstructions using its center by
\[c = \argmin_{c \in \R^n} \max_{k= 1, \dots, K} \norm{\Recon_k \umean - c}  \] and its radius by
\[
r = \max_{k=1, \dots, K} \norm{\Recon_k \umean  -c}.
\]
These in turn can be used as reconstruction value $\Recon \Iop u = c$ and the corresponding radius estimate $\Sradi \Iop u = r$.
An interpretation of these values in the framework of optimal recovery is that every Reconstruction has to lie in the set $\mathcal{V}(\Iop u)$ as defined in section \ref{sec:ourapproach}. Therefore is this bounding sphere is an approximation of the sphere around the information, i.e. the bounding sphere of $\mathcal V(\Iop u)$ and allows us, following Theorem \ref{thm:optirad}, to estimate the error of an optimal recovery. 

\subsection{Mixed recovery/radius indicator methods}\label{sub:mixrr}
In the methods we propose the error of the reconstruction is related to the amount of entropy dissipation, as a bad recovery provoked by a shock should dissipate entropy. 
On the contrary, a bad reconstruction can lead to dissipation where the solution is smooth and we therefore should design tactics to counteract such behavior. 
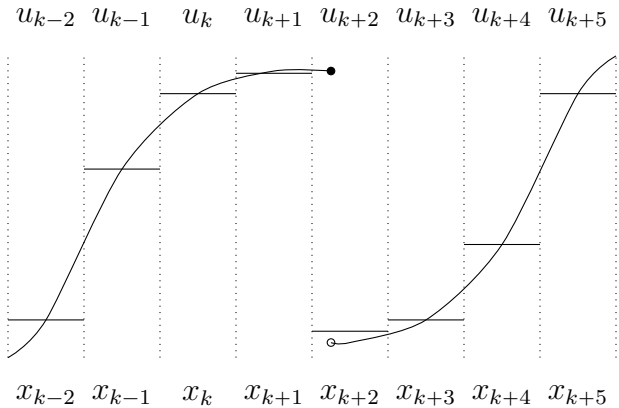
\begin{figure}
	\centering
	\begin{tikzpicture}
		\draw[dotted] (-4.0, -2.0) -- (-4.0, 2.0);
		\draw[dotted] (-3.0, -2.0) -- (-3.0, 2.0);
		\draw [dotted] (-2.0, -2.0) -- (-2.0, 2.0);
		\draw[dotted] (-1.0, -2.0) -- (-1.0, 2.0);
		\draw [dotted] (0.0, -2.0) -- (0.0, 2.0);
		\draw[dotted] (1.0, -2.0) -- (1.0, 2.0);
		\draw [dotted] (2.0, -2.0) -- (2.0, 2.0);
		\draw[dotted] (3.0, -2.0) -- (3.0, 2.0);
		\draw [dotted] (4.0, -2.0) -- (4.0, 2.0);
		
			\draw  plot [smooth] coordinates {(-4.0, -2.0) (-3.5,-1.5) (-2.5,0.5) (-1.5,1.5)  (-0.5, 1.8) (0.25, 1.8)};
			
			\draw plot [smooth] coordinates {(0.25, -1.8) (0.5, -1.8) (1.5, -1.5) (2.5, -0.5) (3.5, 1.5) (4.0, 2.0)};
		
		\filldraw (0.25, 1.8) circle (0.05);
		\draw (0.25, -1.8) circle (0.05);
			
		\draw  (-4.0, -1.5) -- (-3.0, -1.5);
		\draw  (-3.0, 0.5) -- (-2.0, 0.5);
		\draw  (-2.0, 1.5) -- (-1.0, 1.5);
		\draw  (-1.0, 1.77) -- (0.0, 1.77);
		\draw  (0.0, -1.65) -- (1.0, -1.65);
		\draw (1.0, -1.5) -- (2.0, -1.5);
		\draw (2.0, -0.5) -- (3.0, -0.5);
		\draw (3.0, 1.5) -- (4.0, 1.5);
		
		\node  at (-3.5, 2.5) {$u_{k-2}$};
		\node at (-2.5, 2.5) {$u_{k-1}$};
		\node at (-1.5, 2.5) {$u_{k}$};
		\node at (-0.5, 2.5) {$u_{k+1}$};
		
		\node  at (3.5, 2.5) {$u_{k+5}$};
		\node at (2.5, 2.5) {$u_{k+4}$};
		\node at (1.5, 2.5) {$u_{k+3}$};
		\node at (0.5, 2.5) {$u_{k+2}$};
		
		\node  at (-3.5, -2.5) {$x_{k-2}$};
		\node at (-2.5, -2.5) {$x_{k-1}$};
		\node at (-1.5, -2.5) {$x_{k}$};
		\node at (-0.5, -2.5) {$x_{k+1}$};
		
		\node  at (3.5, -2.5) {$x_{k+5}$};
		\node at (2.5, -2.5) {$x_{k+4}$};
		\node at (1.5, -2.5) {$x_{k+3}$};
		\node at (0.5, -2.5) {$x_{k+2}$};

	\end{tikzpicture}
	\caption{Recovery away from discontinuities. Please note that a discontinuity in a solution can be also located in the interior of a cell.}
	\label{fig:discrecov}
\end{figure}

A key element in many reconstruction based methods is the selection of an appropriate stencil for the reconstruction, as for example in \cite{Harten1987ENOIII,LOC1994WENO}. Choosing a non-optimal stencil leads to oscillations in these methods while in our method mainly additional dissipation takes place in this case. Additionally our methods should also benefit from not using bad stencils for recovery. Take for example the situation depicted in Figure \ref{fig:discrecov}. If a recovery for a point value at $x_{k + \frac 1 2}$ is carried out with several different stencils adjacent to this point, all using the same total width,  most of them will recover the value of $u$ at the desired point with acceptable accuracy. Sadly, some recoveries will use a stencil including the cell $k+2$, assuming a continuous function where the function is discontinuous. This will produce a recovery value differing from the more exact value predicted using other stencils, and therefore a high information radius prediction, provoking a lot of dissipation. We will therefore enhance our algorithms with outlier detection - i.e. the ability to discard some of the recoveries - to allow sensible recoveries in the presence of discontinuities. The outlier detection and removal tactics used is
\begin{itemize}
	\item \underline{ Sphere surface discard for the bounding sphere indicator.} Given $K$ recovery operators, the result of one of them has to lie on the surface of the bounding sphere. The result of this operator can be discarded. This process can be repeated $k \ll K$ times.
\end{itemize}
It is important to note, especially with the bounding sphere error indicator, that discarding all but one recovery results in zero predicted recovery error - i.e. the information radius indicator becomes senseless. It is therefore important to limit the number of discarded recoveries as only more than one recovery allow for a sensible prediction.

\subsection{Efficiently solving the optimization problem}
Evaluating the Dafermos flux boils down to solving a constrained optimization problem. For simple fluxes like the flux of Burgers' equation this problem can be solved by pen and paper, especially because the connection to Godunovs flux is known.

\begin{example}[Dafermos flux for Burgers' equation] The optimizer is given for Burgers' equation and compact convex $\uzul$ as
\[
u^*(\uzul, u_l, u_r) = \begin{cases} 
	\begin{cases} 
		0 & 0 \in \uzul \\
		\min \uzul & 0 < \min \uzul \\ 
		\max \uzul & \max\uzul < 0
	\end{cases} & u_l < u_r \\ 
	\begin{cases} 
		\max \uzul & \abs{\max \uzul} > \abs{\min \uzul} \\
		\min \uzul & \abs{\max \uzul} < \abs{\min \uzul}
	\end{cases} & u_l > u_r .
\end{cases}
\]
This follows directly from the representation developed by Osher \cite{Tadmor1984I,Osher1984Riemann}.
But one can also see by direct calculation that this is a solution to this optimization problem. In the case $u_l < u_r$ follows, because $\derd U u$ is monotone as $U$ is convex, that $\derd U u (u_l) < \derd U u (u_r)$ holds. One can therefore distinguish between two cases. 
If $u_l < u_r$ holds we are searching for minimum of $f(u)$. Vice versa, if $u_r < u_l$ is satisfied a maximizer of $f(u)$ on the set $\uzul$ is needed. 
In the first case, if $0$ lies in $\uzul$ this is the minimum sought after, as this is the unconstrained minimizer of $f$. If $0 \not \in \uzul$ applies then the solution has to lie on the boundary of $\uzul$ and we therefore select the end of the interval with minimal absolute value. 

In the second case, we search for the maximum $f(u)$ and there exists therefore no unconstrained solution, i.e. the solution has to lie on one of the ends of $W$ and we select the one with the higher absolute value. 
\end{example}

In the general systems case, finding a solution is not so simple. Apart from using the Lax-Friedrichs modified flux one could in this case use numerical optimization algorithms at every flux evaluation, which is costly. 

\subsection{Distributing viscosity}
 
Numerical experiments with long integration times exposed problems which could be considered the appearance of high frequency modes. When presented with a problem having a smooth solution our schemes solve these satisfactory for small to intermediate integration times, but the predicted reconstruction errors grow several order of magnitude during the runtime. After longer integration times these errors start to cause visible dissipation while the modes themselves are never visible to the bare eye, destroying the high order of the scheme. We will therefore

\begin{itemize}
	\item sketch arguments for the excitation of these high frequency modes
	\item devise a solution by applying even more dissipation to our schemes, but at the correct point in time.
\end{itemize}

To understand the generation of these modes that appear with the exact solution to the optimization problem and also the approximate solution by the modified Lax-Friedrichs flux we will rewrite the approximate flux  into a viscosity form \cite{Harten1978The,HARTEN1983Aclass,Tadmor1984I,Tadmor1984II}
\[
\begin{aligned}
f^{\mathrm{MLF}}_{\uzul}(u_l, u_r) &= f(u_c) + \left(\derd U u (u_l) - \derd U u (u_r) \right) \underbrace{\frac{\max_{u \in \uzul}\norm{\derd f u}\radi (\uzul)}{\norm{\derd U u (u_l) - \derd U u (u_r)}}}_{\mu} \\
&= f(u_c) +  \left(\derd U u (u_l) - \derd U u (u_r) \right) \mu.
\end{aligned}
\]

Here $\mu$ shall be the numerical viscosity applied by the scheme. This viscosity coefficient is not exactly the same as in the classical literature but serves the same needs as it is a scaling factor between the jump of the entropy variables and the added diffusion to a base flux. Classical three point finite volume schemes can to some extend be categorized via their numerical viscosity coefficients and their viscosity has to be bounded as they will otherwise not be monotone for finite time step sizes \cite{HARTEN1983Aclass,Harten1978The,Tadmor1984I,Tadmor1984II}.

\begin{figure}
	\centering
	\begin{tikzpicture}
		\draw [dotted] (-4.0, -2.0) -- (-4.0, 2.0);
		\draw [dotted] (-2.0, -2.0) -- (-2.0, 2.0);
		\draw [dotted] (0.0, -2.0) -- (0.0, 2.0);
		\draw [dotted] (2.0, -2.0) -- (2.0, 2.0);
		\draw [dotted] (4.0, -2.0) -- (4.0, 2.0);
		\draw  (-4.0, -2.0) .. controls (0.0, 2.0)  .. (4.0, -1.5);
		\draw (-4.0, -1.0) -- (-2.0, -1.0);
		\draw (-2.0, 0.6) -- (0.0, 0.6);
		\draw (4.0, -0.5) -- (2.0, -0.5);
		\draw (2.0, 0.55) -- (0.0, 0.55);
		\draw [dashed] (-4.0, -2.0) .. controls (-2.0, 0.0) .. (0.0, 1.0);
		\draw [dashed] (-4.0, -2.0) .. controls (0.0, 1.8) .. (2.0, 0.0);
		\draw [dashed] (-2.0, 0.5) .. controls (2.0, 0.8) .. (4.0, -2.0);
		\node  at (-3.0, 2.0) {$u_{ll}$};
		\node at (-1.0, 2.0) {$u_{l}$};
		\node at (1.0, 2.0) {$u_r$};
		\node at (3.0, 2.0) {$u_{rr}$};
		
		\node  at (-3.0, 0.0) {$x_{ll}$};
		\node at (-1.0, 0.0) {$x_{l}$};
		\node at (1.0, 0.0) {$x_r$};
		\node at (3.0, 0.0) {$x_{rr}$};

	\end{tikzpicture}
	\caption{Counterexample for a bounded viscosity coefficient}
	\label{fig:brokenf}
\end{figure}
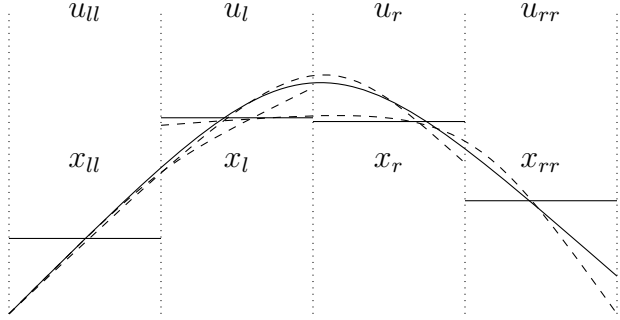

On the contrary, our viscosity distributions lack these boundedness. Assume a smooth solution $u$ has a maximum and two cells next to this maximum have nearly the same average value as cells $x_l$ and $x_r$ in Figure \ref{fig:brokenf}. If the solution is not part of the space considered for reconstructions several different reconstructions will predict different values for the solution at the interface and it is only natural to therefore assume that a sensible error predictor gives a value above $0$ at the interface between $x_l$ and $x_r$. The mean values $u_l$ and $u_r$ on the other hand can be arbitrarily close and therefore also the entropy variables $\derd U u (u_l)$ and $\derd U u (u_r)$. If we therefore look at the viscosity coefficient $\mu$ it is clear that this quantity can be arbitrarily large because the denominator vanishes. 

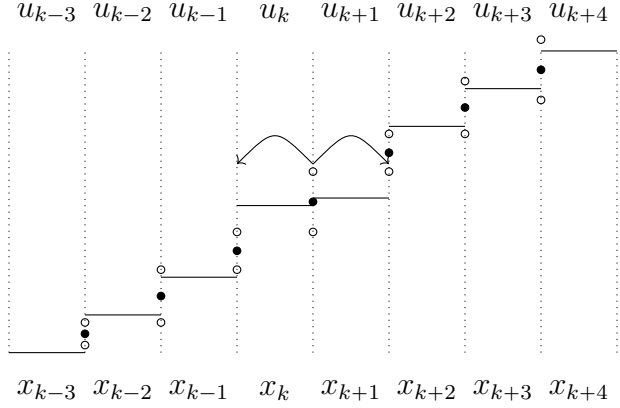
\begin{figure}
	\centering
	\begin{tikzpicture}
			\draw[dotted] (-4.0, -2.0) -- (-4.0, 2.0);
				\draw[dotted] (-3.0, -2.0) -- (-3.0, 2.0);
			\draw [dotted] (-2.0, -2.0) -- (-2.0, 2.0);
			\draw[dotted] (-1.0, -2.0) -- (-1.0, 2.0);
			\draw [dotted] (0.0, -2.0) -- (0.0, 2.0);
			\draw[dotted] (1.0, -2.0) -- (1.0, 2.0);
			\draw [dotted] (2.0, -2.0) -- (2.0, 2.0);
			\draw[dotted] (3.0, -2.0) -- (3.0, 2.0);
			\draw [dotted] (4.0, -2.0) -- (4.0, 2.0);
			
			\draw  [->](0.0, 0.5) .. controls (0.5, 1.0)  .. (1.0, 0.5);
			\draw [->] (0.0, 0.5) .. controls (-0.5, 1.0) .. (-1.0, 0.5);
			\draw  (-4.0, -2.0) -- (-3.0, -2.0);
			\draw  (-3.0, -1.5) -- (-2.0, -1.5);
			\draw  (-2.0, -1.0) -- (-1.0, -1.0);
			\draw  (-1.0, -0.05) -- (0.0, -0.05);
			\draw  (0.0, 0.05) -- (1.0, 0.05);
			\draw (1.0, 1.0) -- (2.0, 1.0);
			\draw (2.0, 1.5) -- (3.0, 1.5);
			\draw (3.0, 2.0) -- (4.0, 2.0);
			
			\filldraw (-3.0, -1.75) circle (0.05);
			\draw (-3.0, -1.9) circle (0.05);
			\draw (-3.0, -1.6) circle (0.05);
			
			\filldraw (-2.0, -1.25) circle (0.05);
			\draw (-2.0, -1.6) circle (0.05);
			\draw (-2.0, -0.9) circle (0.05);
			
			\filldraw (-1.0, -0.65) circle (0.05);
			\draw (-1.0, -0.9) circle (0.05);
			\draw (-1.0, -0.4) circle (0.05);
			
			\filldraw (0.0, 0.0) circle (0.05);
			\draw (0.0, -0.4) circle (0.05);
			\draw (0.0, 0.4) circle (0.05);
			
			\filldraw (1.0, 0.65) circle (0.05);
			\draw (1.0, 0.9) circle (0.05);
			\draw (1.0, 0.4) circle (0.05);
			
			\filldraw (2.0, 1.25) circle (0.05);
			\draw (2.0, 1.6) circle (0.05);
			\draw (2.0, 0.9) circle (0.05);
			
			\filldraw (3.0, 1.75) circle (0.05);
			\draw (3.0, 2.15) circle (0.05);
			\draw (3.0, 1.35) circle (0.05);

			\node  at (-3.5, 2.5) {$u_{k-3}$};
			\node at (-2.5, 2.5) {$u_{k-2}$};
			\node at (-1.5, 2.5) {$u_{k-1}$};
			\node at (-0.5, 2.5) {$u_{k}$};
			
			\node  at (3.5, 2.5) {$u_{k+4}$};
			\node at (2.5, 2.5) {$u_{k+3}$};
			\node at (1.5, 2.5) {$u_{k+2}$};
			\node at (0.5, 2.5) {$u_{k+1}$};
			
			\node  at (-3.5, -2.5) {$x_{k-3}$};
			\node at (-2.5, -2.5) {$x_{k-2}$};
			\node at (-1.5, -2.5) {$x_{k-1}$};
			\node at (-0.5, -2.5) {$x_{k}$};
			
			\node  at (3.5, -2.5) {$x_{k+4}$};
			\node at (2.5, -2.5) {$x_{k+3}$};
			\node at (1.5, -2.5) {$x_{k+2}$};
			\node at (0.5, -2.5) {$x_{k+1}$};

	\end{tikzpicture}
	\caption{Entropy dissipation redistribution to cells with a high jump in the entropy variables. The centers of the recovery are marked as filled circles, while their outer boundaries are marked by empty circles.  }
	\label{fig:viscdis}
\end{figure}

Our solution to this problem could be described as a redistribution of viscosity - and at the same time - entropy dissipation. The justification for this procedure is the observation that errors in a reconstruction appear not only at discontinuities but also in their neighborhood, while dissipation in the analytical setting can only happen when the smoothness of the solution is lost. A different view is the observation that a vanishing jump between $u_l$ and $u_r $ also leads to a vanishing entropy dissipation for fixed viscosity compared to the flux with zeros viscosity, as 
\begin{equation} \label{eq:dispsplit}
	\begin{aligned}
	\skp{\derd U u (u_r) - \derd U u(u_l)}{f(u_c) + \mu \left (\derd U u (u_l) - \derd U u(u_r) \right )} \\
	= \skp{\derd U u (u_r) - \derd U u(u_l)}{f(u_c) } - \mu \norm{\derd U u (u_r) - \derd U u(u_l)}^2
	\end{aligned}
\end{equation}
shows. Viceversa, a small change in the flux over a cell edge, and hence a small amount of viscosity, can lead to strong entropy dissipation if the jump in the (entropy) variables is big enough, as the dissipation scales quadratic with the jump in the entropy variables. In view of this, in a situation as depicted in figure \ref{fig:viscdis} where a high error at the interface $k + \frac 1 2$ would dictate a high amount of viscosity, only a small amount of entropy dissipation takes place, i.e. as $u_k$ and $u_{k+1}$ are average values that lie near to each other, is the entropy dissipation is scaled back. On the other hand can a small change in viscosity at a cell boundary with a high jump in the entropy variables, for example between $x_{k+1}$ and $x_{k+2}$, leads to a high amount of entropy dissipation. It would therefore be wise to swap high amounts of viscosity at cells with small jumps for a slight increase in viscosity at cells with high jumps. 
 We will redistribute viscosity in the following to achieve smooth and bounded viscosity distributions while at the same time making sure that the solutions are even more dissipative than the solutions without redistributed viscosity. Let $\sigma_{l}$ be a positive discrete mollifier kernel centered and symmetric around $0$ with $\sum_l \sigma_l = 1$. We denote the additional amount of entropy dissipation by this viscosity distribution, if used around $k + \frac 1 2$, as
\[
	\begin{aligned}
	\derd {E_\sigma} {t} _{k+\frac 1 2} =& \sum_{l = -p}^p \skp{\derd U u(u_{k + l +1}) - \derd U u(u_{k+l})}{\sigma_l \left (\derd U u (u_{k+l}) - \derd U u (u_{k+l+1}) \right)} \\
	 =& -\sum_{l = -p}^p \sigma_l \norm {\derd U u (u_{k+l+1}) - \derd U u (u_{k+l})}^2. 
	\end{aligned}
\]
Given a viscosity distribution from the modified Lax-Friedrichs flux $\mu_{k+\frac 1 2}$ as defined above we can define a regularized viscosity distribution as
\[
	\begin{aligned}
\left(\sigma \eavs_{\derd U u} \mu \right)_{k+\frac 1 2} = A \sum_{l = -p}^p \sigma_l \frac {-\norm{\derd U u(u_{k + l +1}) - \derd U u(u_{k+l})}^2 \mu_{k+l+\frac 1 2}} {\derd {E_\sigma} {t} _{k+l+\frac 1 2}}, 
\end{aligned}
\]
where $A$ is scaling constant that will be set later.
As an interpretation of this operator we should note the following.
 This operator can be rephrased computationally as a special convolution. The kernel of the convolution is just the mollifier $\sigma_l$, but the viscosity distribution $\mu$ is not directly the second function in this mollification. Instead, the ratio between the entropy dissipation for a $\delta_l = \begin{cases}1, & l = 0 \\ 0, & \text{else} \end{cases}$ viscosity distribution and the total entropy dissipation for a smooth $\sigma$-shaped viscosity distribution is entered into the mollification, allowing us to proof the following lemmas. The first one states that this viscosity redistribution does not reduce the dissipation, i.e. a scheme using the redistributed viscosity is at least as dissipative as the original scheme.

\begin{lemma}[The viscosity distributor is entropy dissipation monotone]
	The entropy aware viscosity smoother is a nonlinear operator satisfying
	\[
	\derd{E^{SV}}{t} \leq \derd{E^{MLF}}{t}
	\] if $A \geq 1$ is chosen and $\sum_l \sigma_l = 1$ holds. 
\end{lemma}
\begin{proof}
	As the center values of $u_c$ in the definition of the fluxes are not changed we can reside to just argue over the change in the entropy dissipation/production incured by the viscosity, i.e. we are omitting the first term on the right hand side of equation \eqref{eq:dispsplit}. A calculation for $A =1$ shows
	\[
		\begin{aligned}
		&\sum_{k } \skp{\derd U u (u_{k+1}) - \derd U u (u_k)}{\mu_{k+1} \left(\derd U u (u_{k}) - \derd U u (u_{k+1})\right)}\\
		 =& -\sum_k\mu_{k+\frac 1 2} \norm{\derd U u (u_{k+1}) - \derd U u(u_k)}^2 \\
		=&  -\sum_k\underbrace{\frac {\sum_{l = -p}^p \sigma_l \norm {\derd U u (u_{k+l+1}) - \derd U u (u_{k+l})}^2 }{\sum_{l = -p}^p \sigma_l \norm {\derd U u (u_{k+l+1}) - \derd U u (u_{k+l})}^2 } }_{=1} 	\mu_{k + \frac 1 2} \norm{\derd U u (u_{k+1}) - \derd U u(u_k)}^2   \\
		=& -\sum_k\sum_{l = -p}^p\frac { \sigma_l \norm {\derd U u (u_{k+l+1}) - \derd U u (u_{k+l})}^2 }{ \derd {E_\sigma} {t} _{k +\frac 1 2}}  	\mu_{k+ \frac 1 2} \norm{\derd U u (u_{k+1}) - \derd U u(u_k)}^2.
		\end{aligned}
	\]
	 	The next step consists of a summation index transform $m = k + l$, and a swapping of the two squared norms
	\[
	\begin{aligned}
		=& -\sum_{m = k + l} \sum_{l = -p}^p\frac { \sigma_l \norm {\derd U u (u_{m+1}) - \derd U u (u_{m})}^2 }{ \derd {E_\sigma} {t} _{m -l +\frac 1 2}}  	\mu_{m-l+ \frac 1 2} \norm{\derd U u (u_{m-l+1}) - \derd U u(u_{m-l})}^2 \\
		=&- \sum_{m = k + l} \sum_{l = -p}^p\frac { \sigma_l  \norm{\derd U u (u_{m-l+1}) - \derd U u(u_{m-l})}^2 	\mu_{m - l+ \frac 1 2}}{ \derd {E_\sigma} {t} _{m -l +\frac 1 2}}   \norm {\derd U u (u_{m+1}) - \derd U u (u_{m})}^2 \\
		=& -\sum_{m = k + l}\left(\sigma \eavs_{\derd U u} \mu \right)_{m + \frac 1 2} \norm {\derd U u (u_{m+1}) - \derd U u (u_{m})}^2 \\
		=&\sum_{m = k+ l}\skp{\derd U u (u_{m+1}) - \derd U u (u_{m})}{ \left(\sigma \eavs_{\derd U u} \mu \right)_{m+\frac 1 2} \left(\derd U u (u_{m}) - \derd U u (u_{m+1})\right)}, 
		\end{aligned}
	\]
	as this allows us to rewrite, using the symmetry of $\sigma_l$, the sum over $l$ into the redistributed viscosity.
	This redistributed viscosity has exactly the same entropy dissipation as the undistributed viscosity. For $A  > 1$ this implies, as the viscosity grows in this case, that the scheme is even more dissipative.
\end{proof}
The second favorable property is that our redistributed viscosity is bounded under mild conditions on the the information radius indicator $\Sradi$. We will see later that in fact these properties are true for the entropy viscosity distributor considered.

\begin{lemma}[The viscosity distributor bounds viscosity ] \label{lem:viscbound}
	The mean values $\Iop u$ are assumed to be of bounded variation. Let $\Recon$ be a reconstruction operator on stencil $S$ with an error indicator $\Sradi$. 
	Assume $(\Recon, S, \Sradi)$ satisfies
	\[
		\Sradi(u) \leq C_1 \TV_S \Iop u
	\]
	together with a mollification kernel $\sigma$ that, when centered on $S$, satisfies
	\[
		\min_{l \in S} \sigma_l > C_2.
	\]
	If the entropy variables used satisfy
	\[
		C_3 \norm{u_l - u_r} \leq \norm{\derd U u (u_l) - \derd U u (u_r)} \leq C_4 \norm{u_l - u_r},
	\]
	and the highest speed in the system is bounded
	\[
		\sup_u \norm{\derd f u } \leq C_5,
	\]
	then the viscosity distribution is bounded from above by $C =A C_1 C_2^{-1} C_3^{-2} C_4 C_5$, where $A$ is the scaling constant from the construction of the viscosity distributor.
	\end{lemma}

\begin{proof}
	We start by acknowledging that while viscosity is itself unbounded, the product of viscosity and entropy variable jump is bounded when the error indicator is.
	\[
		\begin{aligned}
		\norm{\derd U u (u_l) - \derd U u (u_r)}^2 \mu_{k+\frac 1 2} &= \norm{\derd U u (u_l) - \derd U u (u_r)} \radi(u) \sup_u \norm{\derd f u}\\
		 &\leq (\TV_S \Iop u)^2 C_1 C_4 C_5.
		\end{aligned}
	\]
	The entropy dissipation by a $\sigma$ bump formed viscosity on the stencil $S$ given by  $ \derd {E_\sigma} {t} _{l + \frac 1 2}$ can on the other hand be bounded from below using the constant $C_3$ and the total variation
	\[
	 \derd {E_\sigma} t _{l + \frac 1 2} \geq C_3^2 \sum \norm{u_{l+1} - u_l}^2 \sigma_l \geq C_3^2 C_2 \sum \norm{u_{l+1} - u_l}^2 = C_3^2 C_2 (\TV_S \Iop u)^2.
	\]
	We can therefore conclude
	\[
	\left(\sigma \eavs_{\derd U u} \mu \right)_{k+\frac 1 2} = A \sum_{l = -p}^p \sigma_l \frac {\norm{\derd U u(u_{k + l +1}) - \derd U u(u_{k+l})}^2 \mu_{k+l+\frac 1 2}} {\derd {E_\sigma} {t} _{k+l+\frac 1 2}}\leq A \frac {C_1 C_4 C_5}{C_2 C_3^2}
	\]
	as $\sum_l \sigma_l = 1$ sums to 1.
	\end{proof}

	The first condition in the previous lemma is easily verified for the information radius estimators presented in this section when built onto sensible base reconstructions. We will show this result here as an example for the bounding sphere based information radius estimator applied to a scalar conservation law and a linear base reconstruction operator $\Recon_k$, exact at least for constants. As the operator is exact for constants we can split the average values on the stencil $\Iop u = (\Iop u - \bar{\Iop u}) + \bar{\Iop u}$ into their average value and the variation around this average value. As linear finite-dimensional operators the base reconstructions are bounded and all possible norms are equivalent,
	\[
		\norm{(\Recon_k)_k \Iop u} \leq \norm{(\Recon_k)_k} \norm{\Iop u} \implies \norm{(\Recon_k)_k (\Iop u - \bar {\Iop u})} \leq  \norm{(\Recon_k)_k} \norm{\Iop u - \bar {\Iop u}}.
	\]
	Therefore for the error indicator it holds
	\[
		\begin{aligned}
		\Sradi \Iop u =& \frac 1 2 \left (\max_k (\Recon_k \bar \Iop u + \Recon_k (\Iop u - \bar \Iop u))  - \min_k (\Recon_k \bar \Iop u + \Recon_k(\Iop u - \bar \Iop u) ) \right) \\
					\leq &  \frac 1 2 \left (\norm{\bar \Iop u} + \norm{(\Recon_k)_k}C \TV{\Iop u} -  \left(\norm{\bar \Iop u} - \norm{(\Recon_k)_k}C \TV \Iop u \right)  \right ) \\
					=& C \norm{\Recon_k} \TV \Iop u,
		\end{aligned}
	\]
	because the deviation from the average value for a vector of average values can be estimated using the total variation of this vector of average values. As a last piece we will give a sensible value for the constant $A$. From previous analysis it is clear that a value $A \geq 1$ results in a more dissipative viscosity distribution for the total entropy than the un-regularized viscosity distribution. It is therefore worthwhile to look at further cases bounding the usable values for $A$ from below. Assume the setting from figure \ref{fig:Alowb}, namely a viscosity $\mu$ that is zero with exception of a single edge. As explained in the figure, this viscosity would be redistributed. Therefore, if the original amount of viscosity was high enough at that edge, would the new amount be to small at the edge where the viscosity is needed. We therefore propose 
	\[
		A = \min \frac{1}{\sigma_0}.
	\]
	With this constant $\sigma_l$ is not normed to total sum $1$ but unit height at the center. Therefore, if the jump in the entropy variables is of comparable height and the original viscosity distribution is nonzero only at a single edge, the viscosity distributor adds in additional viscosity around this edge but leaves the viscosity at this edge nearly unchanged.
	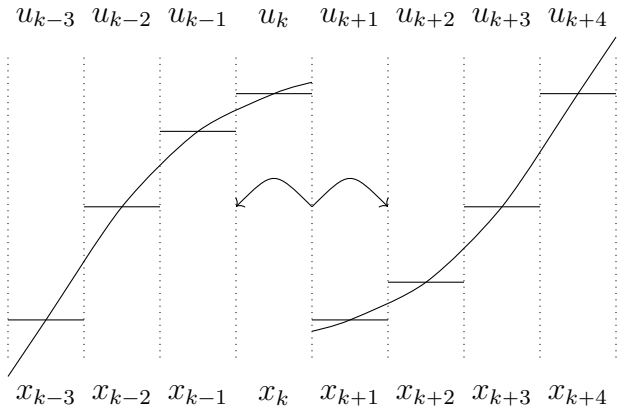
\begin{figure}
		\centering
		\begin{tikzpicture}
			\draw[dotted] (-4.0, -2.0) -- (-4.0, 2.0);
			\draw[dotted] (-3.0, -2.0) -- (-3.0, 2.0);
			\draw [dotted] (-2.0, -2.0) -- (-2.0, 2.0);
			\draw[dotted] (-1.0, -2.0) -- (-1.0, 2.0);
			\draw [dotted] (0.0, -2.0) -- (0.0, 2.0);
			\draw[dotted] (1.0, -2.0) -- (1.0, 2.0);
			\draw [dotted] (2.0, -2.0) -- (2.0, 2.0);
			\draw[dotted] (3.0, -2.0) -- (3.0, 2.0);
			\draw [dotted] (4.0, -2.0) -- (4.0, 2.0);
			
			\node  at (-3.5, 2.5) {$u_{k-3}$};
			\node at (-2.5, 2.5) {$u_{k-2}$};
			\node at (-1.5, 2.5) {$u_{k-1}$};
			\node at (-0.5, 2.5) {$u_{k}$};
			
			\node  at (3.5, 2.5) {$u_{k+4}$};
			\node at (2.5, 2.5) {$u_{k+3}$};
			\node at (1.5, 2.5) {$u_{k+2}$};
			\node at (0.5, 2.5) {$u_{k+1}$};
			
			\node  at (-3.5, -2.5) {$x_{k-3}$};
			\node at (-2.5, -2.5) {$x_{k-2}$};
			\node at (-1.5, -2.5) {$x_{k-1}$};
			\node at (-0.5, -2.5) {$x_{k}$};
			
			\node  at (3.5, -2.5) {$x_{k+4}$};
			\node at (2.5, -2.5) {$x_{k+3}$};
			\node at (1.5, -2.5) {$x_{k+2}$};
			\node at (0.5, -2.5) {$x_{k+1}$};
			
			\draw  [->](0.0, 0.0) .. controls (0.5, 0.5)  .. (1.0, 0.0);
			\draw [->] (0.0, 0.0) .. controls (-0.5, 0.5) .. (-1.0, 0.0);
			
			\draw  (-4.0, -1.5) -- (-3.0, -1.5);
			\draw  (-3.0, -0.0) -- (-2.0, -0.0);
			\draw  (-2.0, 1.0) -- (-1.0, 1.0);
			\draw  (-1.0, 1.5) -- (0.0, 1.5);
			\draw  (0.0, -1.5) -- (1.0, -1.5);
			\draw (1.0, -1.0) -- (2.0, -1.0);
			\draw (2.0, 0.0) -- (3.0, 0.0);
			\draw (3.0, 1.5) -- (4.0, 1.5);
			
			\draw  plot [smooth] coordinates {(-4.0, -2.25) (-3.5,-1.5) (-2.5,0.0) (-1.5,1.0)  (-0.5, 1.5) (0.0, 1.65)};
			
			\draw plot [smooth] coordinates {(0.0, -1.65) (0.5, -1.5) (1.5, -1.0) (2.5, -0.0) (3.5, 1.5) (4.0, 2.25)};
		\end{tikzpicture}
		\caption{Viscosity redistribution in the strong shock case with a low resolution grid. Assume two smooth solutions are joined at $x_{k + \frac 1 2}$ by a discontinuity and that the viscosity calculated just from the radius of information is zero at all edges apart from $x_{k + \frac 1 2}$, and that the viscosity at that edge is just big enough to guarantee only small oscillations. If the viscosity redistributor is used the total entropy dissipation will be the same but the viscosity at $x_{k + \frac 1 2}$, and therefore the local entropy dissipation, could be too small. }
		\label{fig:Alowb}
	\end{figure}

\section{Numerical tests} \label{sec:tests}

After our theoretical observations we will close our presentation with a set of numerical tests for Burgers' equation. Only some of our proposed recoveries/radius indicators $(\Recon, \Sradi)$ are successful in our numerical tests and we will try to analyze the shortcomings of some of the proposed operators and close with a conclusion. The successful operators should be tested on systems and conservation laws in several space dimensions in a future publication. The time integration was carried out using the 8th order method of Prince and Dormand (DP8) as presented by Hairer, N\o rsett and Wanner \cite{HNW2010Solving}. This time integration scheme was chosen as it allows us to use a higher order in time than used in space for all of our test cases. We can therefore test the convergence speed for smooth problems without reducing the time step to ensure the accuracy in time. The ubiquitous family of strong stability preserving Runge-Kutta time integration methods was also tested and while working as satisfactory as the DP8 method we found no advantages from these methods. This could be partially because we can not prove any stability results for explicit Euler steps at the moment. We therefore opted for the higher-order time integration possible using the DP8 method. Through this section a CFL number of $c_{\rm {cfl}} = 0.1$ was used throughout. The source code for the shown numerical experiments is available on GitHub under \href{https://github.com/simonius/EAR}{https://github.com/simonius/EAR}.
\subsection{Tests for Burgers' equation}
    Our first test case is Burgers' equation
    \[
        \der{u}{t} + \der{f}{x} = 0, \text{with } f(u) = \frac{u^2}{2}.
    \]
    We use the initial conditions
    \[
        u_1(x, 0) = \sin(\pi x), \quad u_2(x, 0) = \begin{cases} -1 & x < 1 \\ 1 & x > 1 \end{cases}, \quad u_3(x, 0) = 1 + \frac{\sin(\pi x)}{50}.
    \]
    The first initial condition was selected to demonstrate the ability of the schemes to handle a smooth solution that develops a shock over time and to carry on with the calculation after the onset of the shock. We are especially interested if the scheme produces nonphysical oscillations afterwards. 

    A second test case is initial condition $u_2$, demonstrating the ability of the scheme to handle rarefaction waves. Also, the behavior of the scheme near the sonic point of the flux can be studied using this test case. 
    
    Both initial conditions are run with $N = 50$ cells. While the first test uses periodic boundary conditions on the interval $\Omega = [0, 2]$, the second test case uses the interval $\Omega = [1/2, 3/2]$ with outflow boundary conditions.

    An initial condition with a smooth solution for a long period of time is given by $u_3$ and will be used for a convergence analysis. Using this long integration time we can also study deterioration in the performance of the scheme stemming from excitated high order modes with only small amplification factor. The error of the numerical solution was calculated using backtracking of the characteristics to the initial condition, i.e. the solution of the implicit equation \cite{Lax1973Hyperbolic}
    \[
    	u(x, t) = u_0\left(x - \derd f u (u(x, t)) t, 0\right).
    \]
    We will test these initial conditions using several different combinations of recovery/error indicator procedures. Because the theory concerning the viscosity redistributors is for now hinged onto the Lax-Friedrichs modified flux and this flux can be evaluated significantly easier. This flux being used through all tests in conjunction with the viscosity redistribution. The polynomial order was fixed to $p = 4$, therefore every base reconstruction uses $r = p + 1 = 5$ cells. Because $r + 1$ different stencils are possible it follows that the total stencil is $d = 2 r = 10$ cells wide. The shape function in the entropy redistributor used was scaled to be $w = d + 1 = 11$ cell boundaries wide, i.e. all cells interfaces in the $d$ cell stencil stencil are used in the calculation as dictated by Lemma \ref{lem:viscbound}. As shape function a Hann window, originally used in digital signal processing \cite{BT1958Measurement}, was selected. The value of $A$ is therefore set to $A = 5$.

    \subsubsection{Tests for the variance based information radius indicator}
      In figure \ref{fig:StochOptBurgTests}, the results for the first two initial conditions can be seen in conjunction with polynomial order $ p = 4$ and the variance based information diameter predictor. While the rarefaction wave looks promising, especially free of a sonic point glitch, the solution to the initial condition $u_1$ only looks good up to $t = 0.3$. This can be attributed to the fact that the discontinuity that forms around this time leads to a very localized smearing that also deforms the solution in more than 10 cells away from the shock in a chain reaction. This could be seen as a contagion of a loss of smoothness from cell to cell because bad base reconstructions are used. While this effect is only marginal at $t = 0.6$ its influence is clearly visible at $t = 1.2$. We will see that this problem plagues all of our base information radius predictors which was our motivation to design the second set of predictors explained in subsection \ref{sub:mixrr}. The scheme is able to demonstrate a high order of accuracy as can be seen in figure \ref{fig:StochCA}. Still this order of accuracy is lower than expected for a classical reconstruction based scheme where one would expect a fifth order convergence rate, while here only a fourth order convergence is demonstrated. 
    \begin{figure}
    	\begin{subfigure}{0.49\textwidth}
    		\includegraphics[width=\OptTwidth]{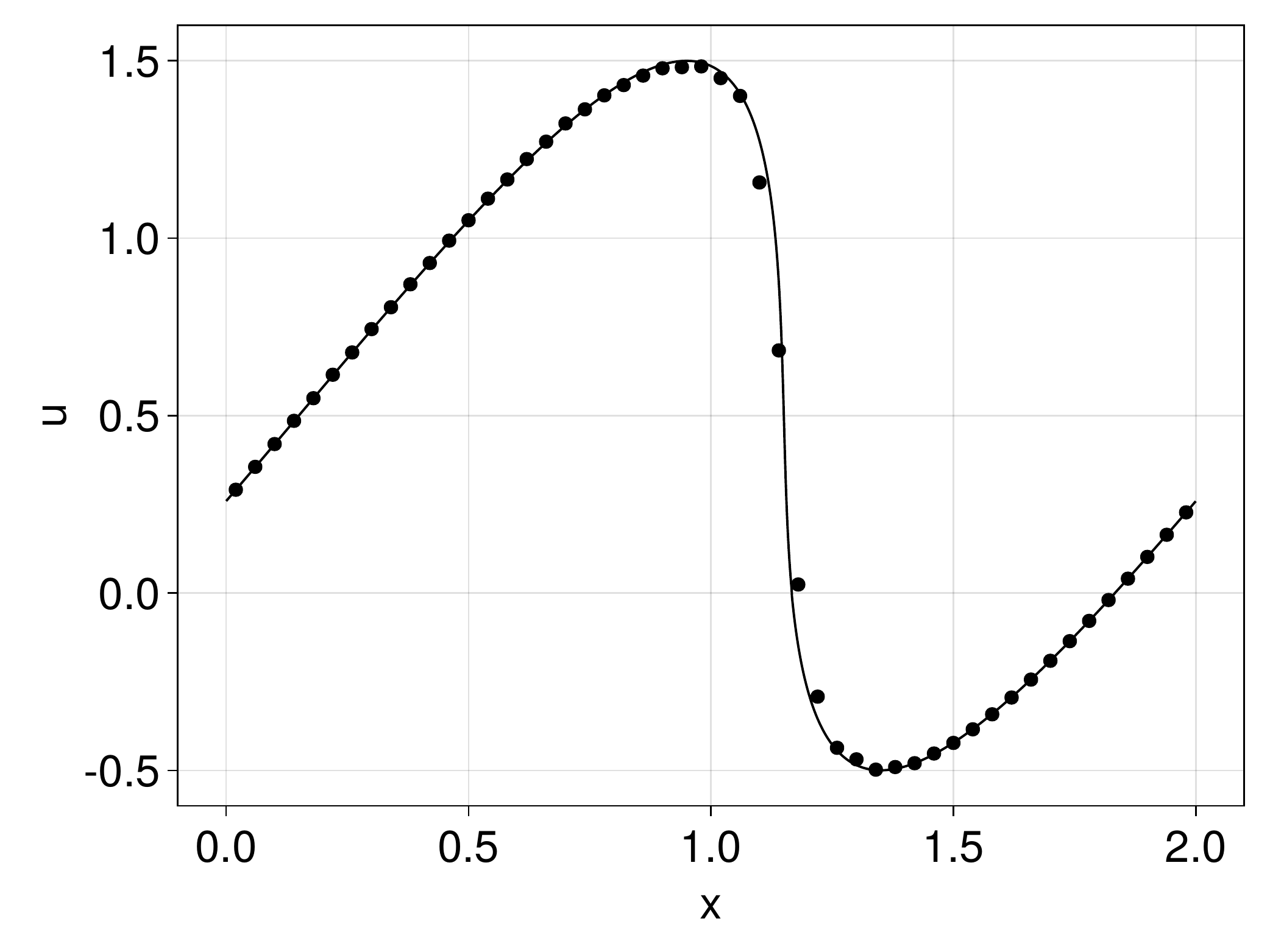}
    		\subcaption{Solution to $u_1(x, 0)$ at $t = 0.3$}
    	\end{subfigure}
    	\begin{subfigure}{0.49\textwidth}
    		\includegraphics[width=\OptTwidth]{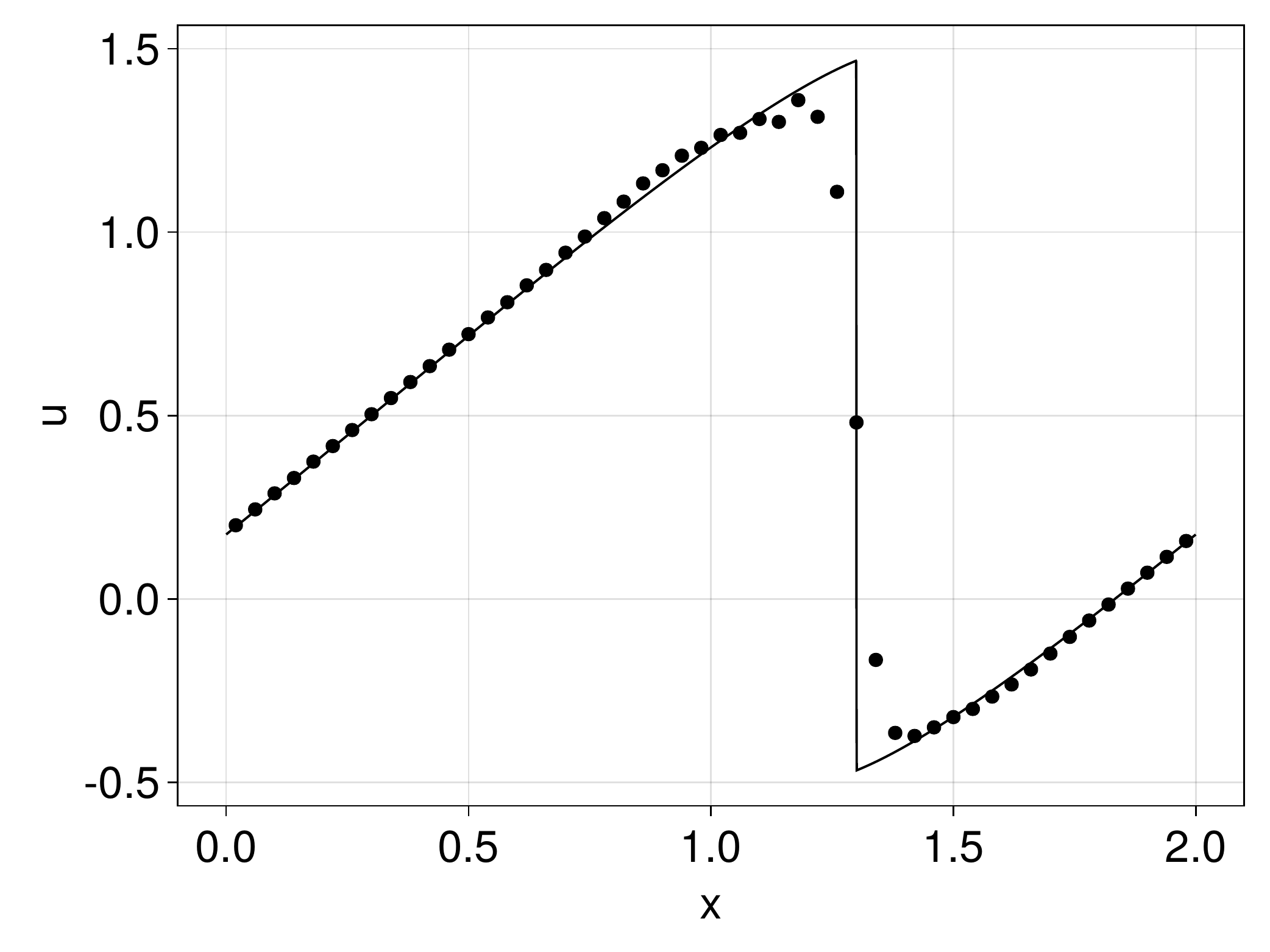}
    		\subcaption{Solution to $u_1(x, 0)$ at $t = 0.6$}
    	\end{subfigure}
    	\begin{subfigure}{0.49\textwidth}
    		\includegraphics[width=\OptTwidth]{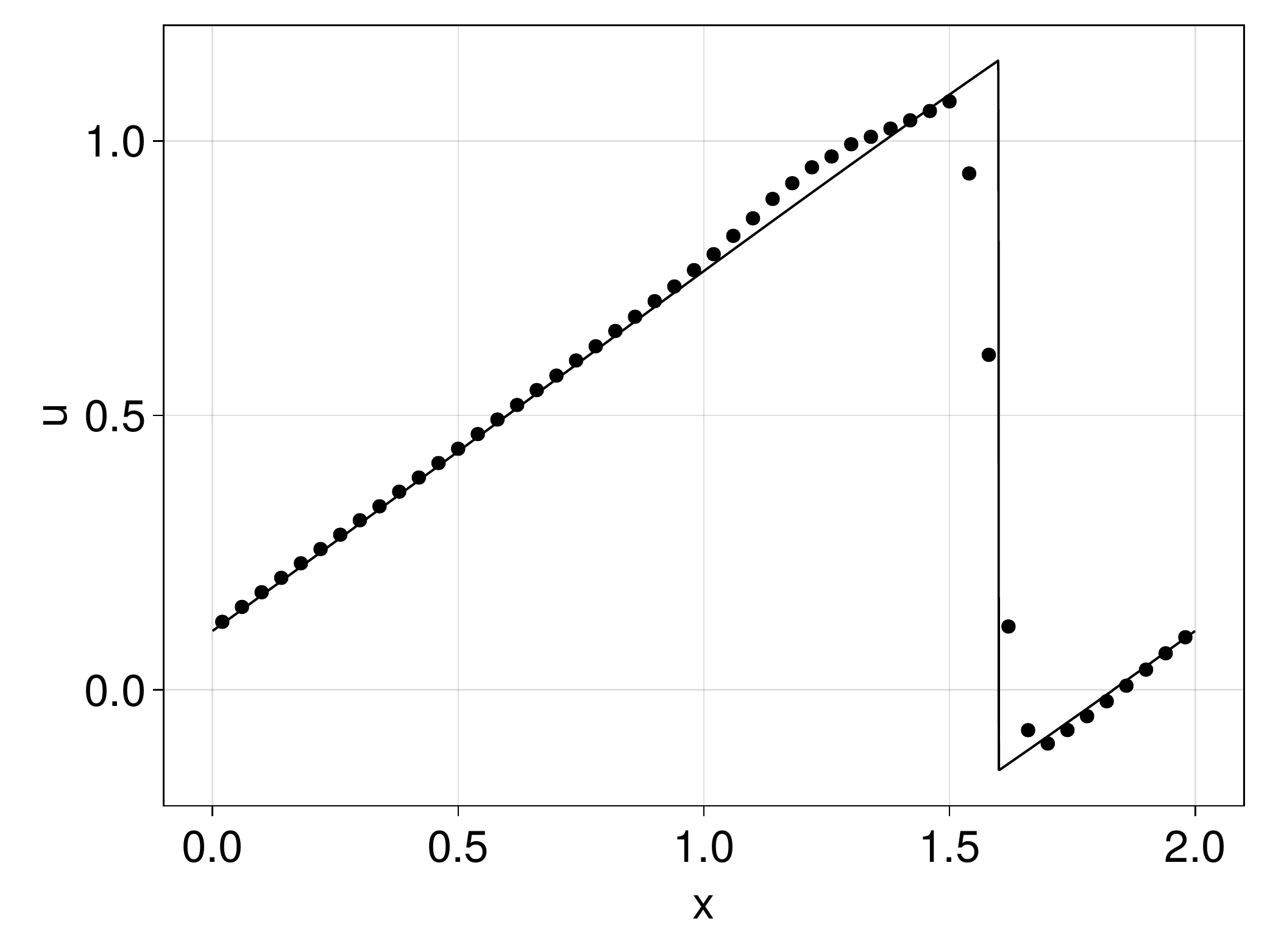}
    		\subcaption{Solution to $u_1(x, 0)$ at $t = 1.2$}
    	\end{subfigure}
    	\begin{subfigure}{0.49\textwidth}
    		\includegraphics[width=\OptTwidth]{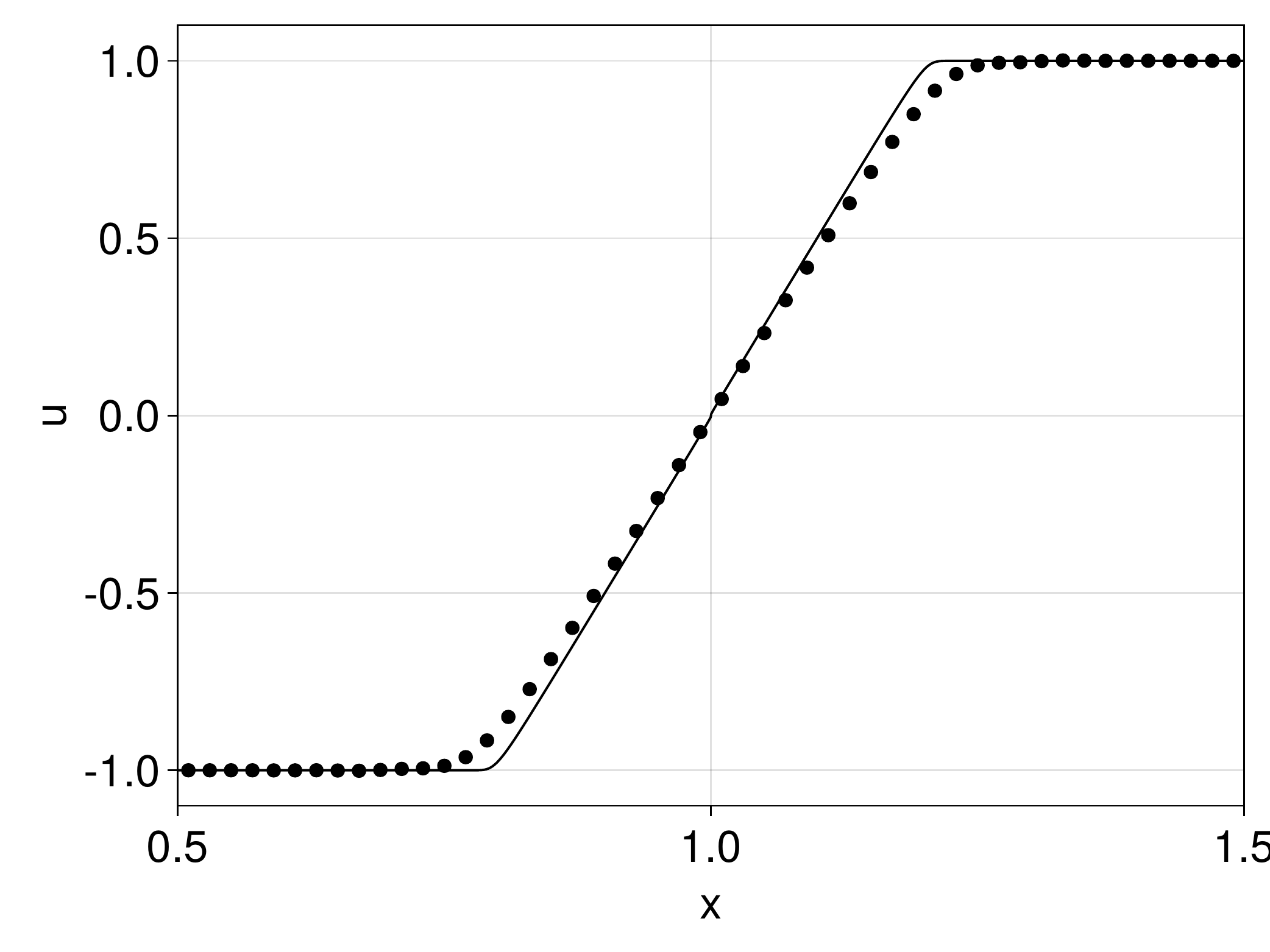}
    		\subcaption{Solution to $u_2(x, 0)$ at $t = 0.2$}
    	\end{subfigure}
    	\caption{Solution to Burgers' equation using the variance based information diameter predictor.}
    	\label{fig:StochOptBurgTests}
    \end{figure}
	 \begin{figure}
	 	\begin{subfigure}{0.48\textwidth}
			\centering
			\includegraphics[width=0.9\textwidth]{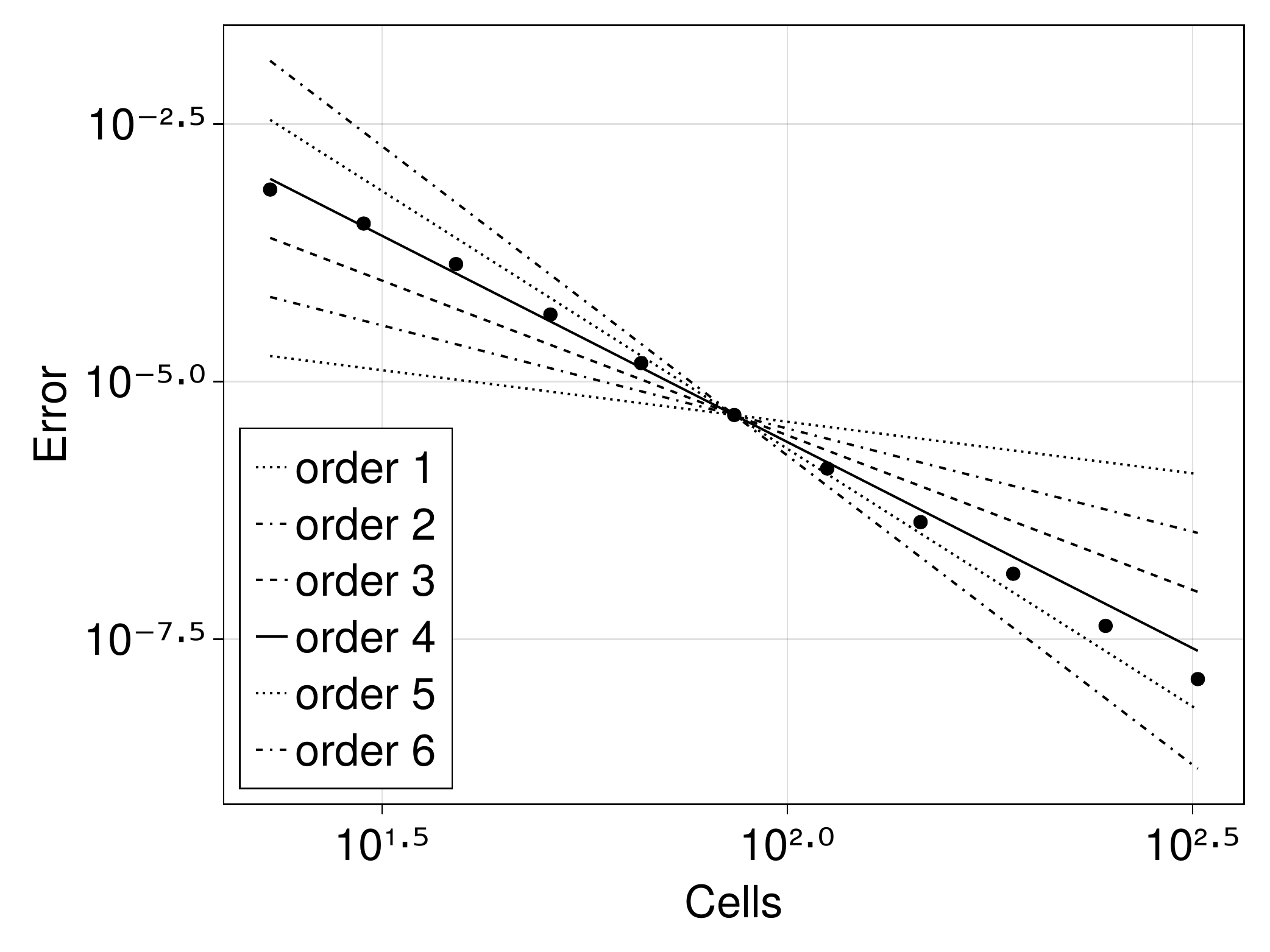}
			\caption{Convergence analysis for polynomial order p=4 and the variance based information diameter predictor.}
			\label{fig:StochCA}
		\end{subfigure}
		\begin{subfigure}{0.48\textwidth}  
				\centering
				\includegraphics[width=0.9\textwidth]{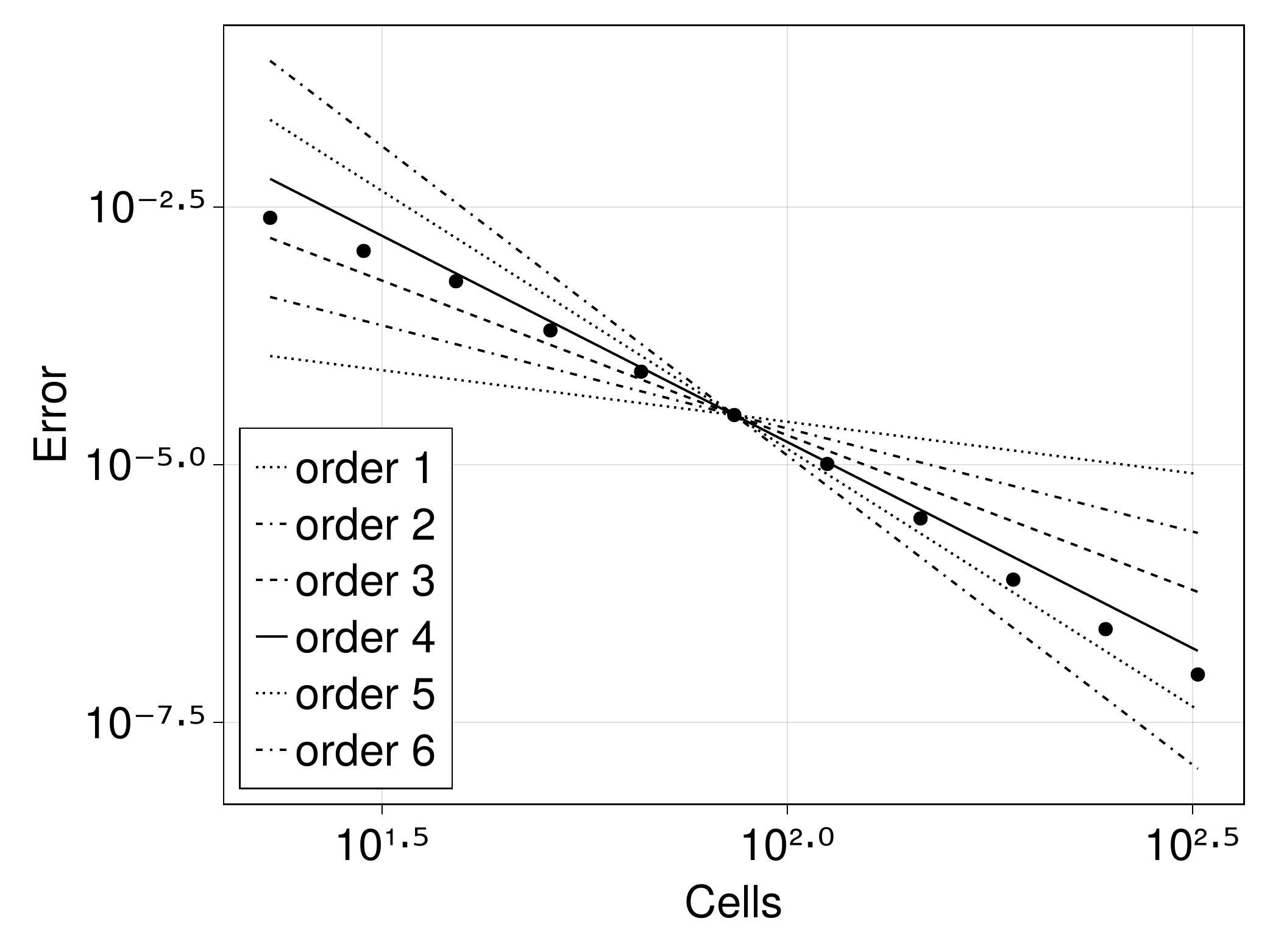}
				\caption{Convergence analysis for polynomial order p=4 and the bounding sphere information diameter predictor.}
				\label{fig:BBCA}	
		\end{subfigure}
		\begin{subfigure}{0.48\textwidth}
					\centering
					\includegraphics[width=0.9\textwidth]{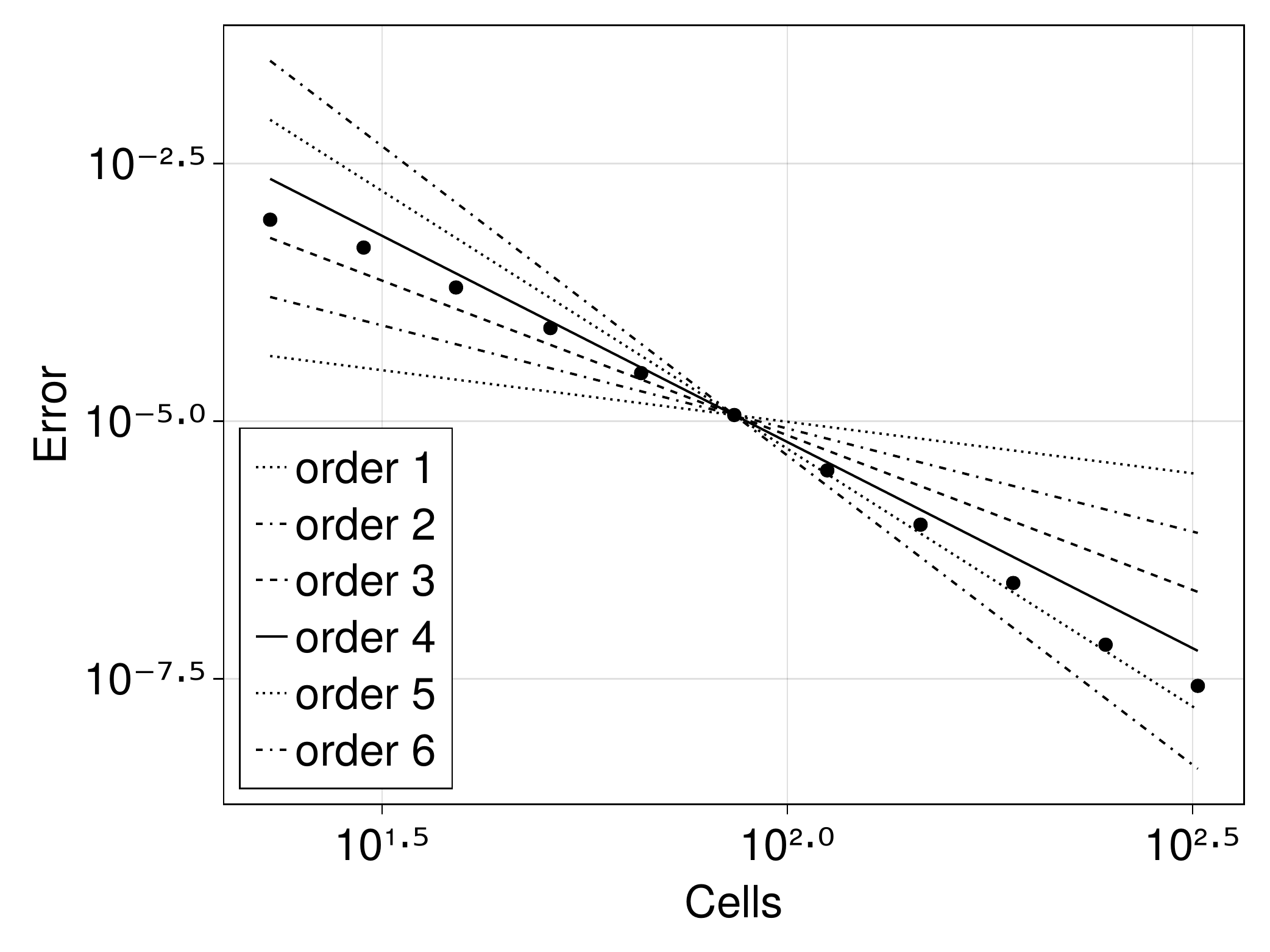}
					\caption{Convergence analysis for polynomial order p=4 and the bounding box surface discard information diameter predictor.}
					\label{fig:ODBBCA}
		\end{subfigure}
		\begin{subfigure}{0.48\textwidth}
			\includegraphics[width=0.9\textwidth]{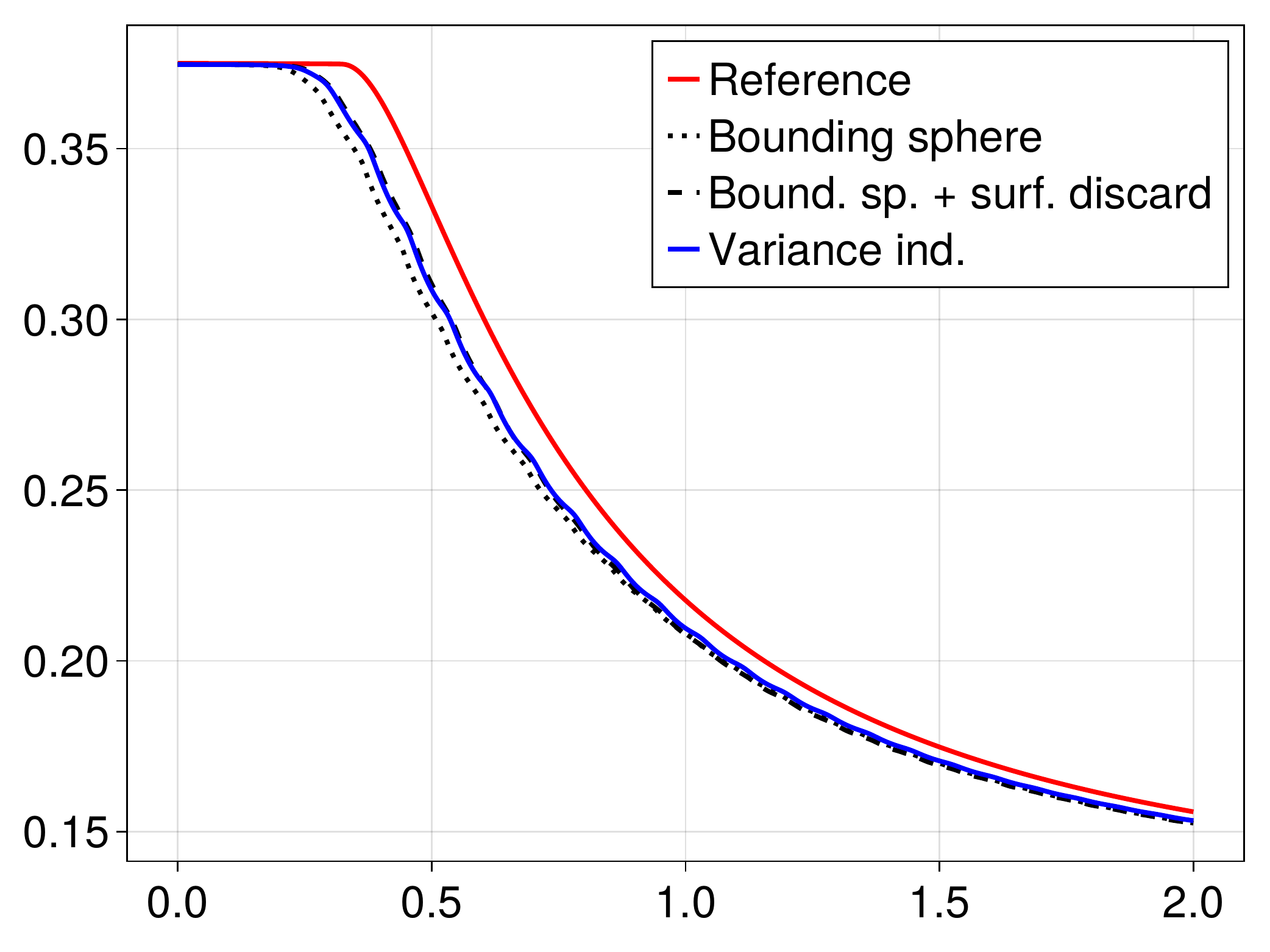}
			\caption{Total entropies of all tested solvers for the initial condition $u_1$ and $N = 50$ cells. Obviously, all entropies lie below the reference entropy calculated from the solution of a Godunov scheme with $N = 5000$ cells.}
		\end{subfigure}
		\caption{Convergence analyses and total entropy analysis for all tested combinations.}
	\end{figure}
    \subsubsection{Tests for the bounding sphere error indicator}
   The results for the bounding sphere error indicator in figure \ref{fig:BBCA} and \ref{fig:BBoptBurgTests} look comparable to the ones for the variance based estimator. While the smearing effect is somewhat stronger, oscillations are also damped significantly stronger. Once more, a fourth order convergence can be expected for smooth solutions if polynomials of degree $p = 4$ are used, one degree lower than expected. 
    
    \begin{figure}
        \begin{subfigure}{0.49\textwidth}
        \includegraphics[width=\OptTwidth]{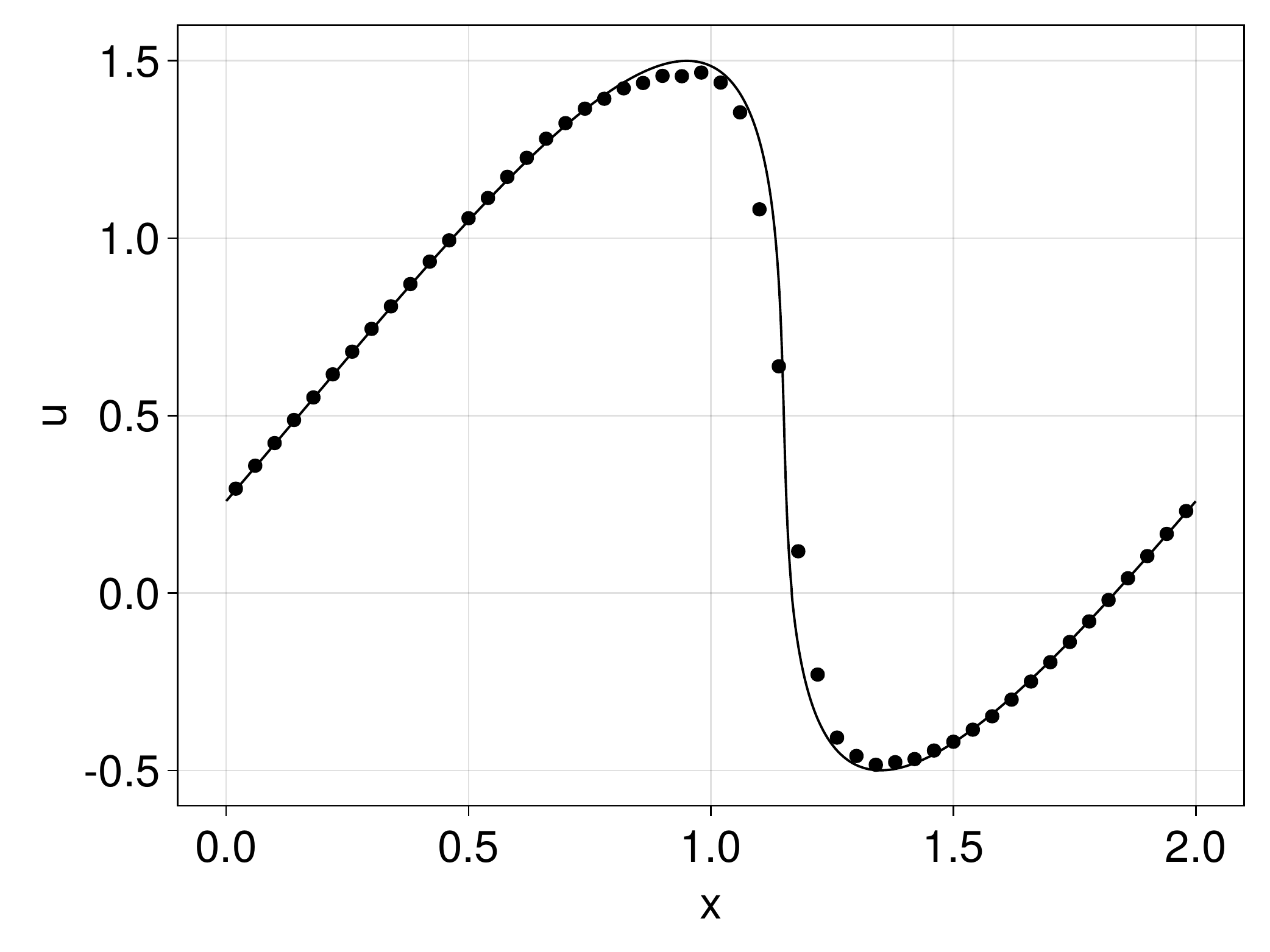}
        \subcaption{Solution to $u_1(x, 0)$ at $t = 0.3$}
        \end{subfigure}
        \begin{subfigure}{0.49\textwidth}
        \includegraphics[width=\OptTwidth]{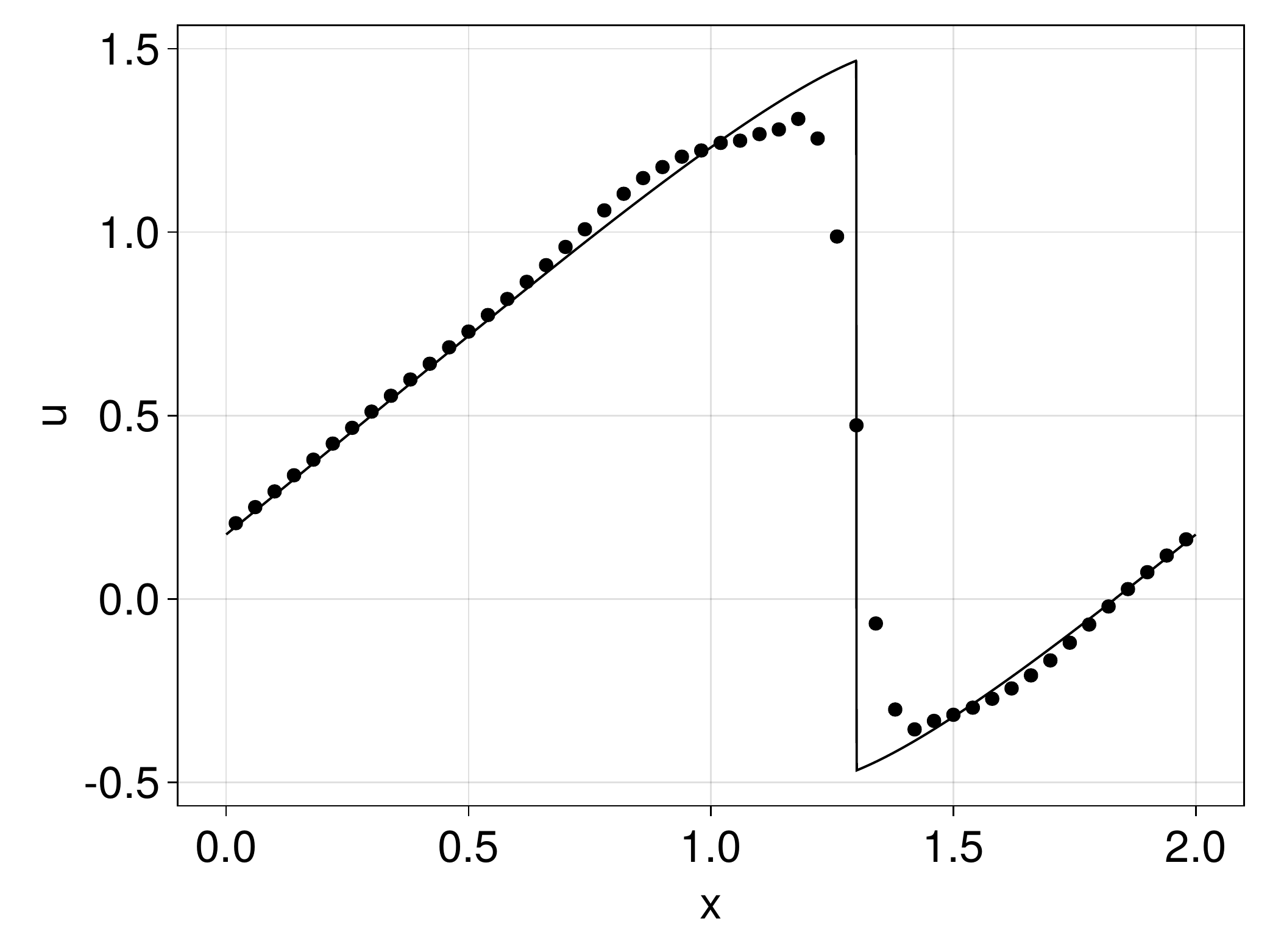}
        \subcaption{Solution to $u_1(x, 0)$ at $t = 0.6$}
        \end{subfigure}
        \begin{subfigure}{0.49\textwidth}
        \includegraphics[width=\OptTwidth]{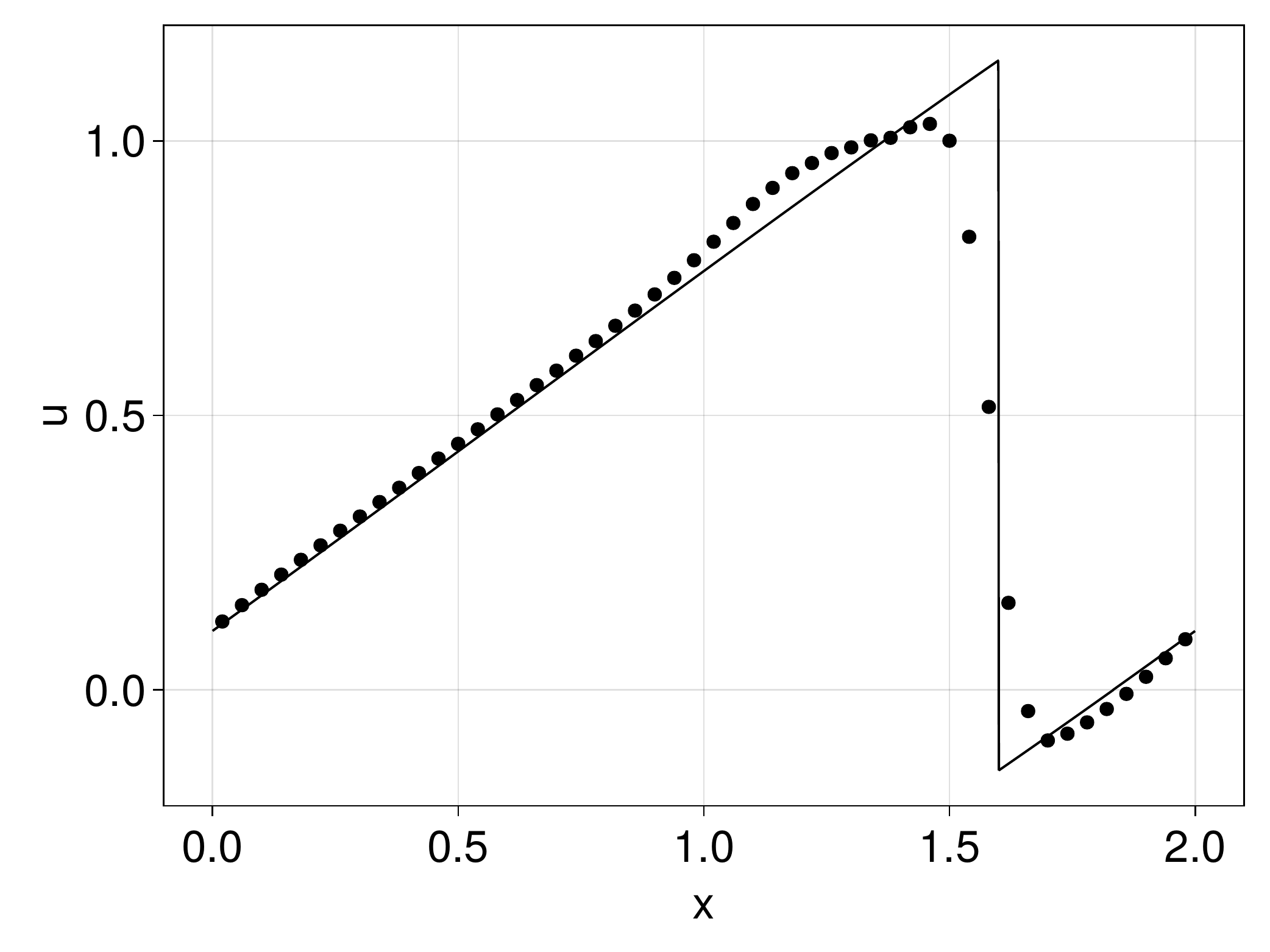}
        \subcaption{Solution to $u_1(x, 0)$ at $t = 1.2$}
        \end{subfigure}
        \begin{subfigure}{0.49\textwidth}
        \includegraphics[width=\OptTwidth]{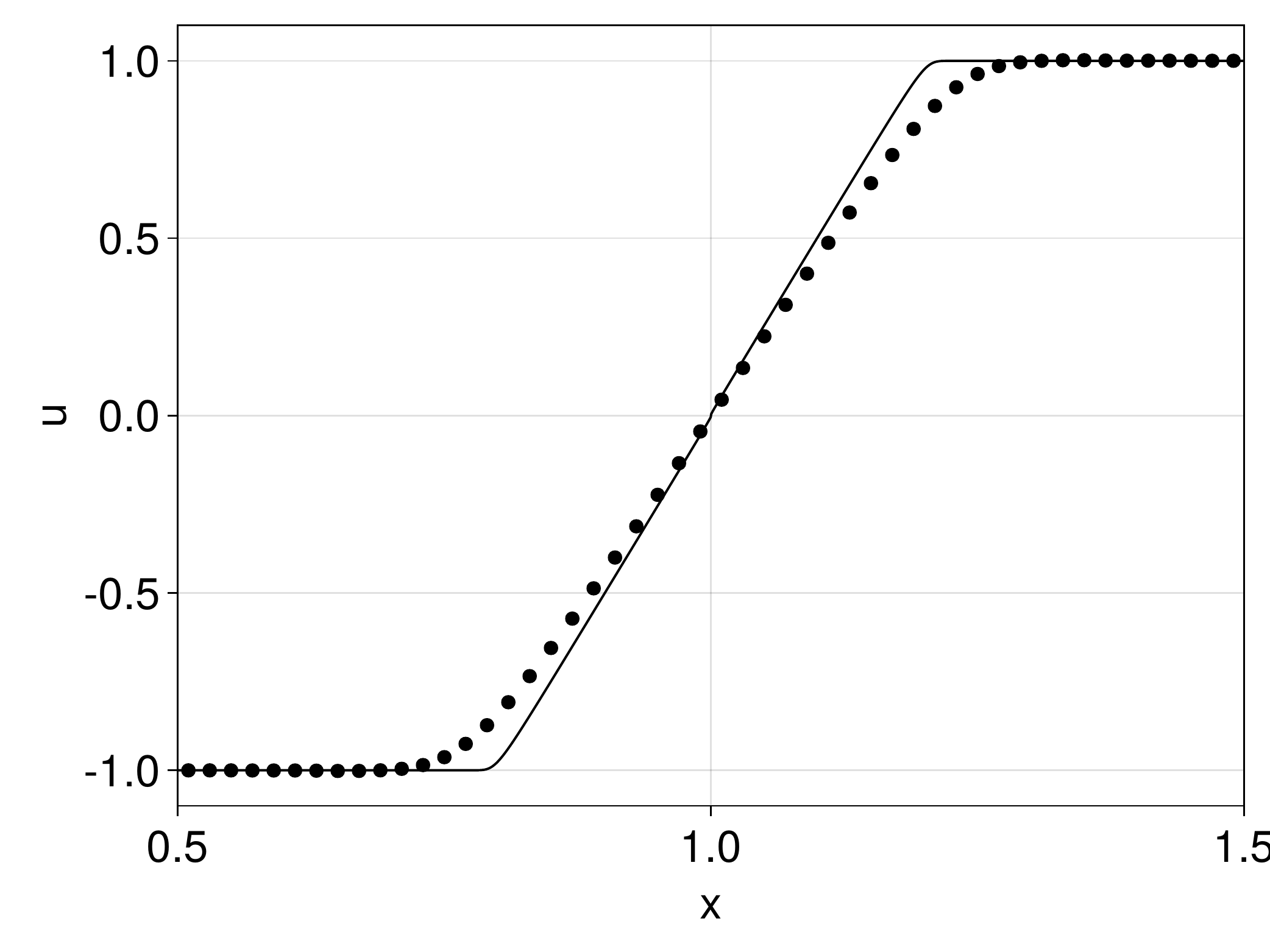}
        \subcaption{Solution to $u_2(x, 0)$ at $t = 0.2$}
        \end{subfigure}
    \caption{Solution to Burgers' equation using the bounding sphere information diameter predictor.}
    \label{fig:BBoptBurgTests}
    \end{figure}
	\subsubsection{Tests for sphere surface discard error indicator.}
		Discarding the recoveries on the surface of their combined bounding sphere as done for the simulations in \ref{fig:ODBBoptBurgTests} and \ref{fig:ODBBCA} is a suitable strategy to improve the shock capturing capabilities of our schemes. The convergence speed of this scheme is also excellent. In this case two recoveries were discarded.
				\begin{figure}
			\begin{subfigure}{0.49\textwidth}
				\includegraphics[width=\OptTwidth]{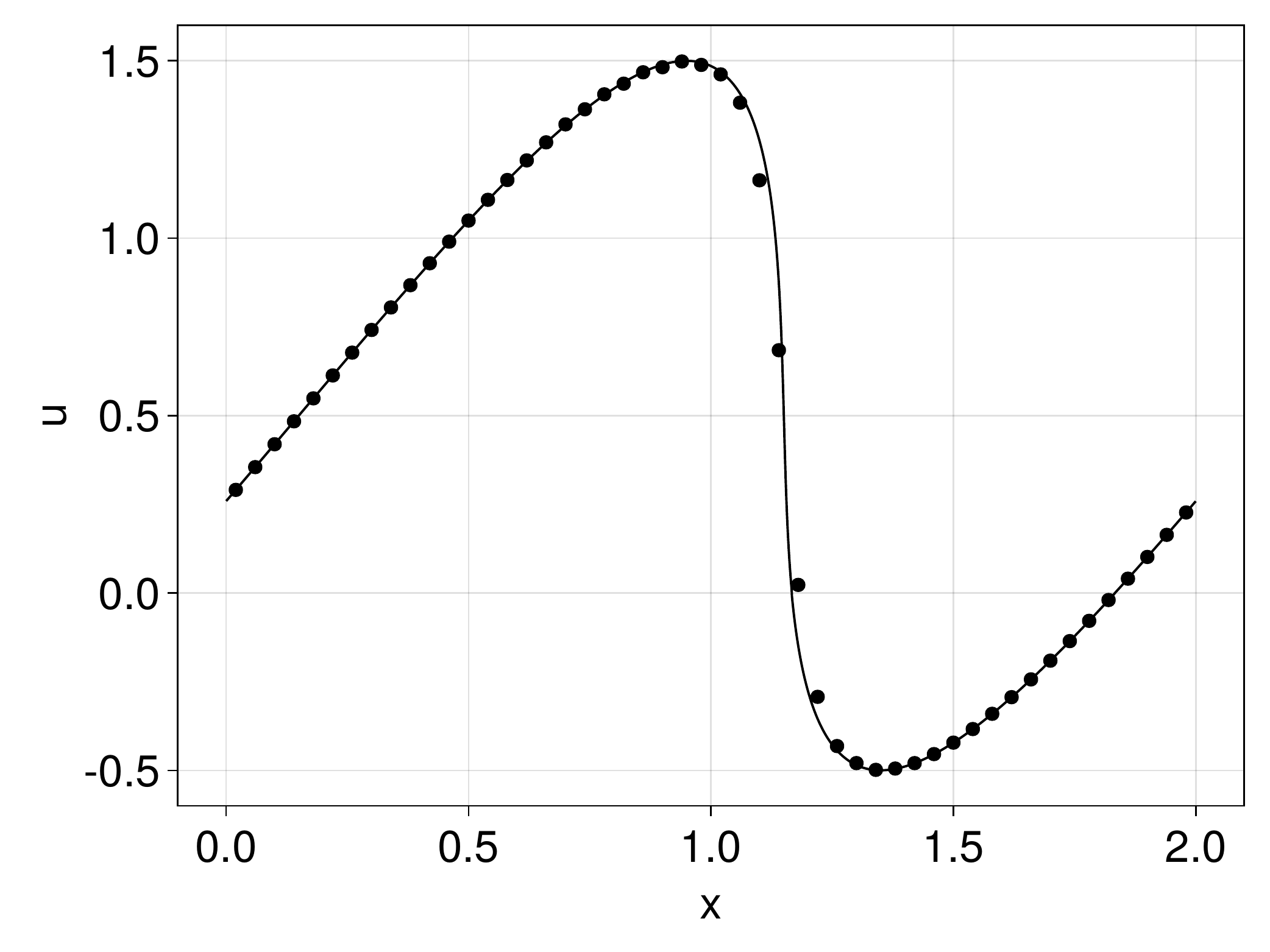}
				\subcaption{Solution to $u_1(x, 0)$ at $t = 0.3$}
			\end{subfigure}
			\begin{subfigure}{0.49\textwidth}
				\includegraphics[width=\OptTwidth]{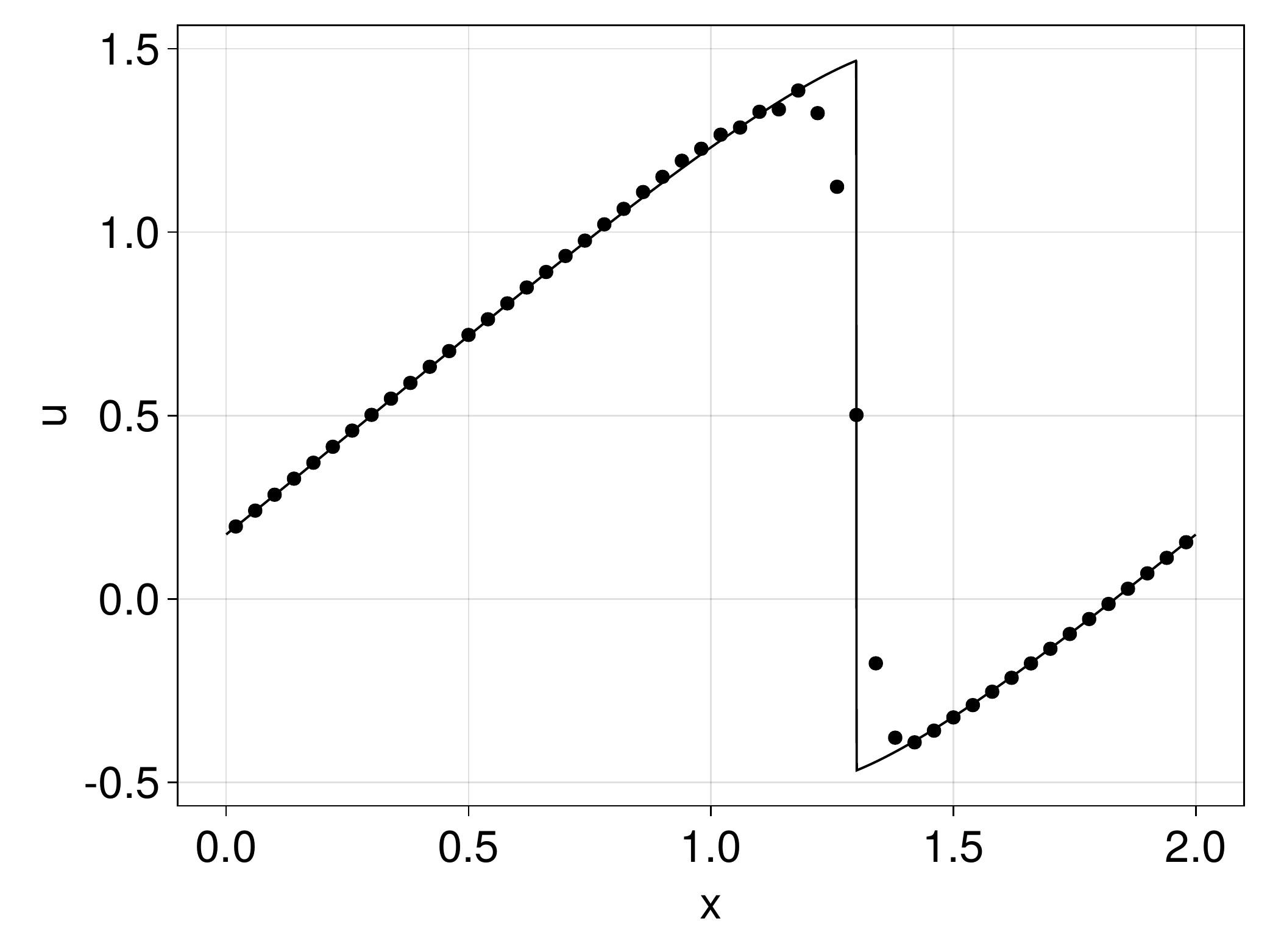}
				\subcaption{Solution to $u_1(x, 0)$ at $t = 0.6$}
			\end{subfigure}
			\begin{subfigure}{0.49\textwidth}
				\includegraphics[width=\OptTwidth]{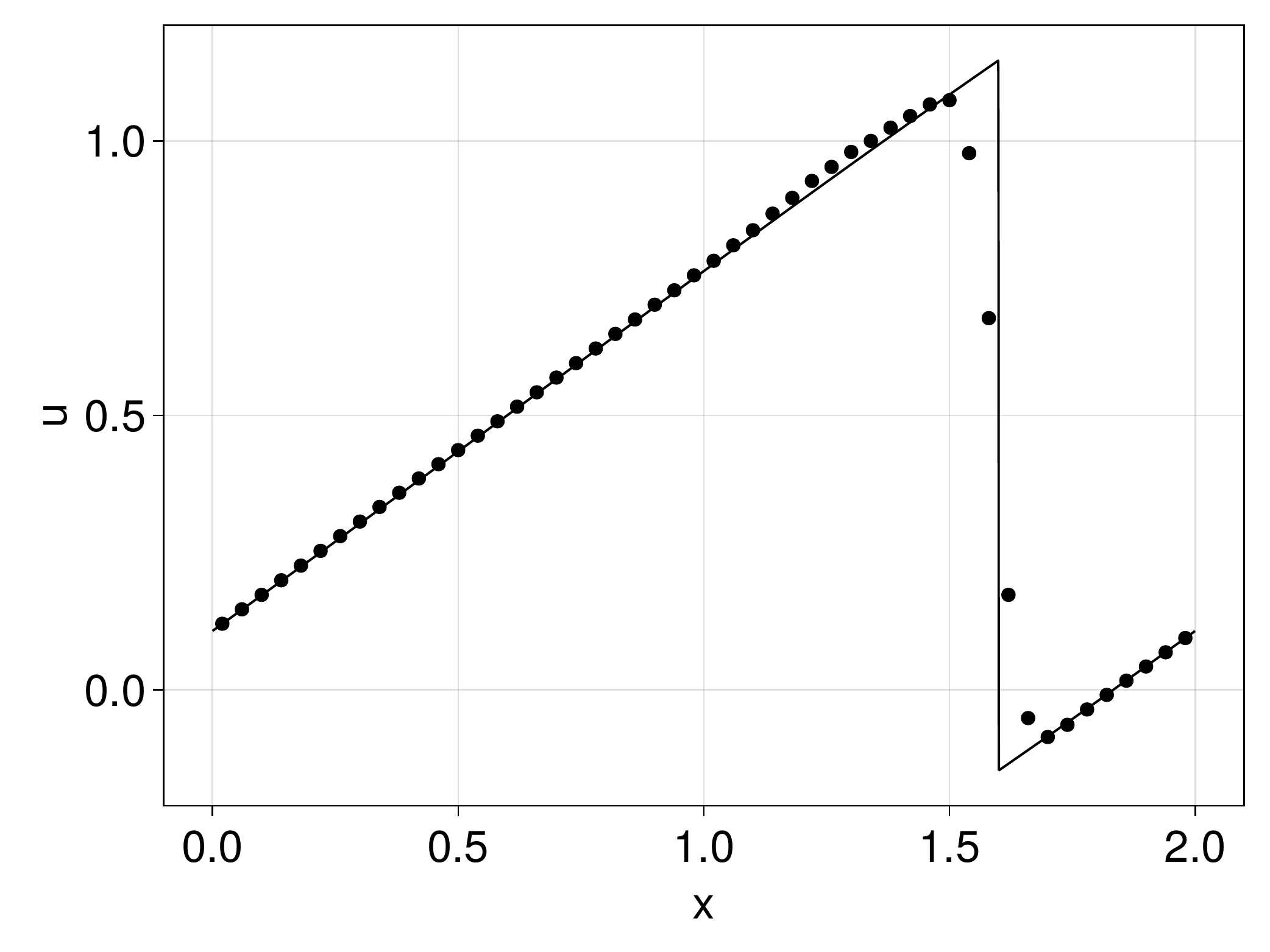}
				\subcaption{Solution to $u_1(x, 0)$ at $t = 1.2$}
			\end{subfigure}
			\begin{subfigure}{0.49\textwidth}
				\includegraphics[width=\OptTwidth]{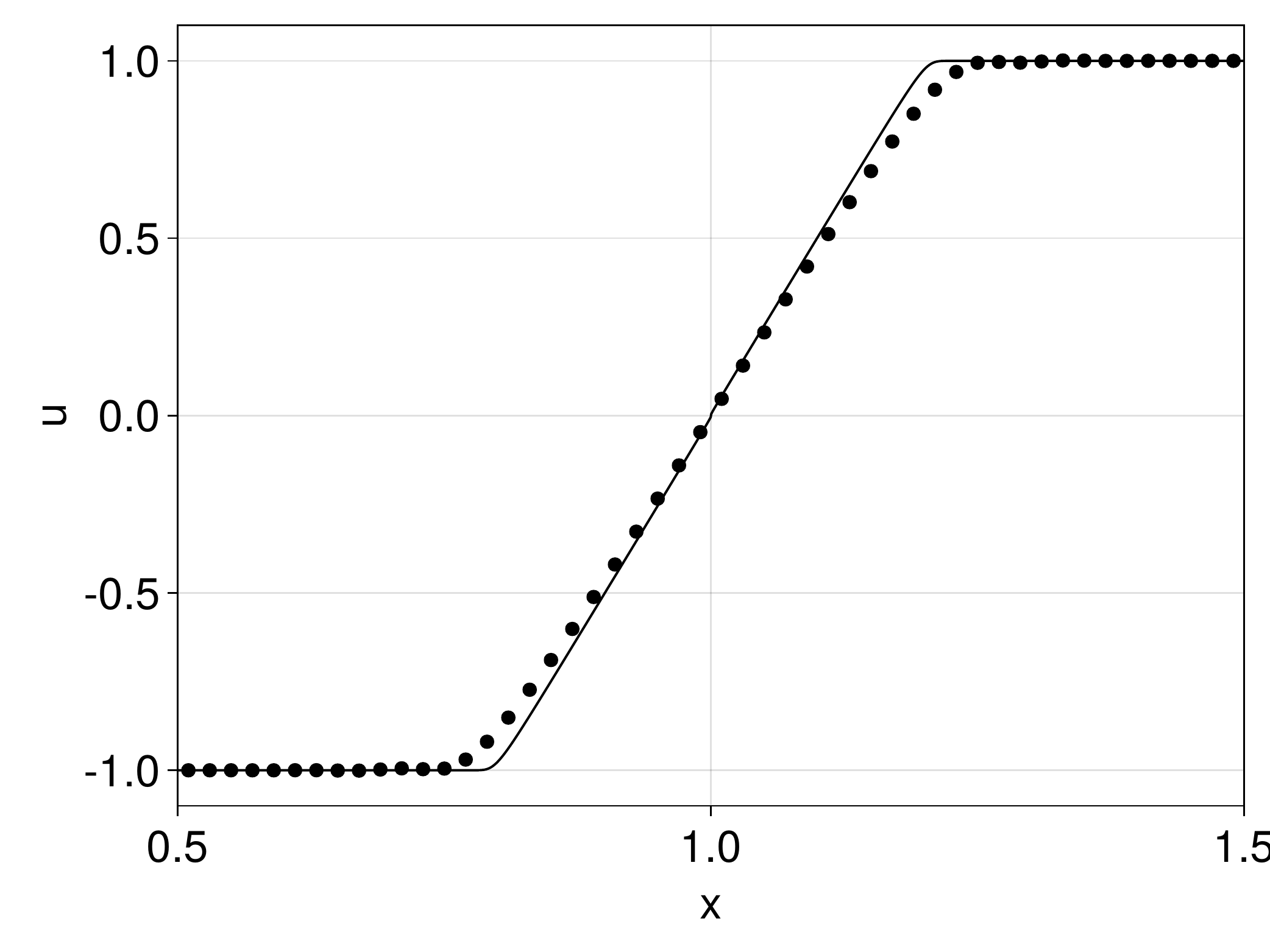}
				\subcaption{Solution to $u_2(x, 0)$ at $t = 0.2$}
			\end{subfigure}
			\caption{Solution to Burgers' equation by the bounding box surface discard information diameter predictor.}
			\label{fig:ODBBoptBurgTests}
		\end{figure}

\section{Conclusions}
In this work a new framework for the design of numerical methods for hyperbolic conservation laws was laid out. This new framework allows the design of efficient high order methods based on Dafermos' entropy rate criterion in conjunction with optimal recovery procedures. This approach is not only novel in the fact that Dafermos' entropy rate criterion, only limited by an error estimate, is enforced directly through the design of numerical fluxes. In fact a second valuable generalization was made for the design of reconstruction based high order methods. Previously, these methods relied on the design of methods recovering a piece wise polynomial function with jumps at the cell interfaces from average values. Afterwards numerical fluxes based on these polynomials, their jumps, and solutions to the underlying PDE with these polynomials as initial condition were constructed. In this work, instead, point values of the solution combined with error estimates are recovered from the given average values, and an appropriate flux approximation is constructed without any knowledge of the solution of the (generalized) Riemann problem. We therefore term these new methods recovery based to differentiate them from reconstruction based methods like ENO and WENO methods. An interesting new line of research is opened by this approach. While classical ENO methods have to consider all possible stencils - a senseless task in several space dimension for higher degrees of accuracy - our method is also able to construct nearly non-oscillatory schemes using a restricted set of stencils, as a less accurate recovery with a restricted stencil set is still possible, and just leads to appropriately higher dissipation. Future work will therefore lead to tests in several space dimensions, but first, a forthcoming paper will present the generalization to hyperbolic systems of conservation laws.
\section*{Acknowledgements}
SK would like to thank Jan Glaubitz for discussions concerning the viscosity distribution problem in Oberwolfach.

\appendix 
\bibliographystyle{plain}
\bibliography{literature}

\end{document}